\documentclass[a4paper,10pt,reqno]{amsart}
\usepackage{amsmath}
\usepackage[T1]{fontenc}
\usepackage{amssymb}
\usepackage{amsthm}
\usepackage{color}

\DeclareMathOperator{\curl}{curl}

\newcommand{\R}{\mathbb{R}}
\newcommand{\Aa}{\mathbf A}
\newcommand{\Ab}{\mathbf A}
\newcommand{\Bb}{\mathbf B}

\newcommand{\Fb}{\mathbf F}
\newtheorem{thm}{Theorem}[section]
\newtheorem{prop}[thm]{Proposition}
\newtheorem{lem}[thm]{Lemma}
\newtheorem{corol}[thm]{Corollary}

\newtheorem{lemma}[thm]{Lemma}

\newtheorem{rem}[thm]{Remark}

\numberwithin{equation}{section}

\title[Transition to normal phase]
{On the transition to the normal phase for superconductors surrounded by normal conductors}

\author[S. Fournais]{S\o ren Fournais}
\address{(S. Fournais and A. Kachmar) Department of Mathematical Sciences, University
  of Aarhus, Ny Munkegade, Building
  1530, DK-8000 \AA rhus C}
\email[S. Fournais]{fournais@imf.au.dk} \email[A.
Kachmar]{ayman.kachmar@math.u-psud.fr}

\author[A. Kachmar]{Ayman Kachmar}

\subjclass[2000]{Primary 35Q55, Secondary 35J25, 35A15, 58E50, 81Q10}
\keywords{Ginzburg-Landau type equations, 
Schr\"o\-dinger operator with magnetic field, semiclassical
analysis}

\begin{document}

\begin{abstract}
For a cylindrical superconductor surrounded by a normal material, we
discuss transition to the normal phase of stable, locally stable and
critical configurations. Associated with those phase transitions, we
define critical magnetic fields and  we provide a sufficient
condition for which those critical fields coincide. In particular,
when the conductivity ratio of the superconducting and the normal
material is large, we show that the aforementioned critical magnetic
fields coincide, thereby proving that the transition to the normal
phase is sharp. One key-ingredient in the paper is the analysis of
an elliptic boundary value problem involving `transmission' boundary
conditions. Another key-ingredient involves a monotonicity result
(with respect to the magnetic field strength)  of the first eigenvalue of a
magnetic Schr\"odinger operator with discontinuous coefficients.
\end{abstract}
\maketitle
\tableofcontents

\section{Introduction}
A type~II superconductor placed in an applied magnetic field
undergoes three phase transitions: When the intensity of the applied
field is below a first critical value $H_{C_1}$, the sample exhibits
the Meissner effect and remains in a superconducting phase. When the
field is increased further past $H_{C_1}$, the sample is in a mixed
state and the magnetic field penetrates the material in quantized
vortices. Increasing the field further past another critical value
$H_{C_2}$, the sample exhibits surface superconductivity, and when
the field is increased above $H_{C_3}$, superconductivity is lost
and the sample is in the normal phase. The above picture is
rigorously established for extreme type~II materials through the
minimization of the Ginzburg-Landau functional, see for instance the
papers \cite{FoHe, FoKa, HePa,  LuPa, Pa, Se} and the books
\cite{FoHe07, SaSe} for  results and additional references
concerning the subject.

In addition to the phase transitions associated with minimizers
(stable states) of the Ginzburg-Landau functional, type~II materials
posses hysteresis associated with local minimizers (locally stable
states) of the energy. For instance, a locally stable state that
does not posses vortices will remain locally stable in increasing
magnetic fields up to a {\it super-heating field}, and a similar
phenomenon is associated with a {\it sub-cooling field} associated
with decreasing applied fields (see \cite{Se99}). It is therefore
natural to address a similar question when dealing with the
transition of normal states: Does local stability persists for the
normal state below $H_{C_3}$, and for the superconducting state
above $H_{C_3}$, or will there be hysteresis? That is, we ask
whether the transition from the superconducting phase to the normal
phase happens at a sharp
critical value of the magnetic field ($=H_{C_3}$).

For type~II superconducting samples with smooth boundaries and
surrounded by the vacuum, Fournais and Helffer \cite{FoHe} showed
that the transition is indeed sharp. Hysteresis is excluded in
\cite{FoHe07}. The case of domains with corners \cite{BoFo} and the
$3$-dimensional case \cite{FoHe, LuPa1} (see also \cite{HeMo2,LuPa1}) have also been studied. The
reason for the sharp transition is essentially the monotonicity of
the first eigenvalue $\mu_1(B)$ of the Neumann Schr\"odinger
operator
$$-(\nabla-iB\mathbf F)^2\quad{\rm in}~\Omega,\quad\Omega\subset\mathbb R^d\quad(d=2,3)\,,$$
for large values of the magnetic field, a property known as strong
diamagnetism (see \cite{FoHe, FoHe06}). Here $\mathbf
F:\overline\Omega\to\mathbb R^d$ is a vector field such that ${\rm
  curl}\,\mathbf F$ is a constant.

In this paper, we address the same question---the
transition from superconducting to normal phase---but for a
superconductor surrounded by normal materials. It is well known from the
superconducting proximity effect (see \cite{deGe, deGe1}) that the
presence of a normal material exterior to a superconductor allows
the superconducting electron Cooper pairs to flow into the normal
material in a narrow boundary layer. The characteristic length scale
of that layer is called the `extrapolation length'. To model this
phenomenon,  one has to consider a generalized Ginzburg-Landau
theory where the order parameter and the magnetic potential are not only
defined in the superconducting material but also in the normal
material surrounding it.

For a cylindrical superconducting sample of cross
section $\Omega\subset\mathbb R^2$, surrounded by a normal material
and placed in a magnetic field parallel to the cylinder axis, the
Gibbs free energy is given by the following Ginzburg-Landau type
functional (see \cite{Chetal}):
\begin{align}\label{GL-func}
\mathcal G(\psi,\Ab)=\mathcal G_{\kappa, H}(\psi,\Ab)&=\int_\Omega\left(|\nabla_{\kappa
H\Ab}\psi|^2-\kappa^2|\psi|^2+\frac{\kappa^2}2|\psi|^4\right)\, d 
x\\
&\quad+\int_{\widetilde\Omega\setminus\Omega}\left(\frac1m|\nabla_{\kappa
H\Ab}\psi|^2+a\kappa^2|\psi|^2\right) d  x\nonumber\\
&\quad+(\kappa H)^2\int_{\widetilde\Omega}|{\rm curl}\, \Ab-1|^2\, d 
x\,.\nonumber
\end{align}
Here we use the notation,
\begin{align}
  \label{eq:21}
  \nabla_{\Ab} = \nabla - i \Ab\,,
\end{align}
for the magnetic gradient.

In the above functional, $\widetilde\Omega\subset\mathbb R^2$ is a
bounded domain such that $\overline\Omega\subset\widetilde\Omega$ and
$\widetilde\Omega\setminus\Omega$ is the cross section of the normal
material\,\footnote{In part of the existing literature on the subject
  (see for example \cite{Kach, Chetal}) $\widetilde\Omega$ is
  taken to be all of $\mathbb R^2$. However, in the original paper of
\cite{deGeHu}, the normal material (i.e.
$\widetilde\Omega\setminus\Omega$) has been taken to be bounded. We
take the latter point of view, in order to avoid certain technical
difficulties which are unimportant for our present purpose. The
functional analysis necessary to overcome those difficulties is
developed in \cite{Gi}.}, the Ginzburg-Landau parameter $\kappa>0$ is a characteristic of the
superconducting material (the ratio of two characteristic lengths), $H>0$ is the intensity of the applied
magnetic field, $a>0$ is a characteristic of the normal material
that depends on the temperature and its sign signifies that the
temperature is above the critical temperature of the normal
material. Finally,  $m>0$ is the conductivity ratio of the
superconducting and normal materials. Minimization of the functional
(\ref{GL-func}) will take place over finite-energy configurations
$(\psi,\Ab)\in H^1(\widetilde\Omega;\mathbb C) \times
H^1(\widetilde\Omega;\mathbb R^2)$. Starting from a minimizing
sequence, it is quite standard to prove the existence of minimizers of
\eqref{GL-func}, see 
\cite{Gi}.

{\bf We will always assume that $\Omega$ and $\widetilde \Omega$ are
  smooth, bounded and simply connected}.

Notice that the normal state $(0,\mathbf F)$, with $\mathbf F$ being
the unique vector field satisfying
\begin{equation}\label{F}
{\rm curl}\,\mathbf F=1\,,\quad{\rm div}\,\mathbf F=0\quad{\rm
  in~}\widetilde\Omega\,,
\quad \nu\cdot\nabla\mathbf F=0\quad{\rm on}~\partial\widetilde\Omega \,,
\end{equation}
is  a critical point of the functional (\ref{GL-func}). It can also
be shown that this state is the unique normal state up to a {\it
gauge transformation} (see \cite{Gi}). Configurations which are
gauge equivalent to the normal state $(0,\mathbf F)$ will be called
{\it trivial} throughout the paper.

Defining the set,
\begin{equation}\label{c-giorgi}
\mathcal N^{\rm sc}(a,m;\kappa)=\{H>0~:~ \mathcal
G_{\kappa, H} {\rm ~has~ non\text{-}trivial ~critical ~points}\}\,,
\end{equation}
then it is  known from \cite[Theorem~4.7]{Gi} (see also
Theorem~\ref{thm-GP-type} in the present paper) that the above set
is bounded.

In connection with stability and local stability of
the normal state $(0,\mathbf F)$, we also
introduce the two sets:
\begin{equation}\label{HC3-loc}
\begin{split}
\mathcal N(a,m;\kappa)&=\{H>0~:~\mathcal G_{\kappa,H}~{\rm has~a~non\text{-}trivial~minimizer}\}\,,\\
\mathcal  N^{\rm
loc}(a,m;\kappa)&=\{H>0~:~\mu^{(1)}(\kappa,H)<0\}\,.
\end{split}
\end{equation}
Here $\mu^{(1)}(\kappa,H)$ is the ground state energy of the quadratic
form
\begin{equation}\label{qf}
\mathcal Q[\kappa,H](\phi)=\int_{\Omega}\left(|\nabla_{\kappa
H\mathbf F}\phi|^2-\kappa^2|\phi|^2\right)dx+
\int_{\widetilde\Omega\setminus\Omega}\left(\frac1{m}|\nabla_{\kappa H\mathbf
F}\phi|^2+a\kappa^2|\phi|^2\right)dx\,,
\end{equation}
i.e.
\begin{equation}\label{EV}
\mu^{(1)}(\kappa,H)=\inf_{\substack{\phi\in
    H^1(\widetilde\Omega)\\\phi\not=0}}\left(\frac{\mathcal
Q[\kappa,H](\phi)}{\|\phi\|_{L^2(\widetilde\Omega)}^2}\right)\,.
\end{equation}
Since $\mathcal Q[\kappa,H]$ defines the Hessian of the functional
$\mathcal G_{\kappa,H}$ at  the normal state $(0,\mathbf F)$, we see
that  if $H\in \mathcal N^{\rm loc}(a,m;\kappa)$, then $(0,\mathbf
F)$ is not a local minimizer $\mathcal G_{\kappa,H}$. Hence we
obtain the following inclusion,
$$\mathcal N^{\rm loc}(a,m;\kappa)\subset\mathcal N(a,m;\kappa)\,. $$
On the other hand,  the following  inclusion is trivial,
\begin{equation}\label{eq-inclusion-hc3}
\mathcal N(a,m;\kappa)\subset\mathcal N^{\rm sc}(a,m;\kappa)\,.
\end{equation}

One of the main results of the present paper is the following.

\begin{thm}\label{thm-main}
Given $a>0$ and $m>1$, there exists $\kappa_0>0$ such that, for all
$\kappa\geq\kappa_0$, the following equalities hold,
$$\mathcal N^{\rm sc}(a,m;\kappa)=
\mathcal N^{\rm loc}(a,m;\kappa)=\mathcal N(a,m;\kappa)\,.
$$
\end{thm}

In the literature, it is typical  to introduce the following critical
fields (cf. e.g. \cite{FoHe,LuPa}),
\begin{align}
&\overline H^{\rm sc}_{C_3}(a,m;\kappa)=\sup \mathcal N^{\rm
  sc}(a,m;\kappa)\,,\quad
\underline H^{\rm sc}_{C_3}(a,m;\kappa)=\inf \mathbb R_+\setminus \mathcal N^{\rm
  sc}(a,m;\kappa)\,,\label{hc3-sc}\\
&\overline H_{C_3}(a,m;\kappa)=\sup \mathcal N(a,m;\kappa)\,,\quad
\underline H_{C_3}(a,m;\kappa)=\inf\mathbb R_+\setminus \mathcal
N(a,m;\kappa)\,,
\label{hc3-gl}\\
&\overline H_{C_3}^{\rm loc}(a,m;\kappa)=\sup \mathcal N^{\rm loc}(a,m;\kappa)\,,\quad
\underline H_{C_3}^{\rm loc}(a,m;\kappa)=\inf\mathbb R_+\setminus \mathcal
N^{\rm loc}(a,m;\kappa)\,,
\label{hc3-loc}
\end{align}

As a corollary of Theorem~\ref{thm-main}, we get a result concerning
equality of the above critical fields.

\begin{corol}\label{thm-main-corol}
Given $a>0$ and $m>1$, there exists $\kappa_0>0$ such that, for all
$\kappa\geq\kappa_0$, the following equalities hold,
$$
\overline H^{\rm sc}_{C_3}(a,m;\kappa)=\overline H
_{C_3}(a,m;\kappa)=\overline H^{\rm loc}_{C_3}(a,m;\kappa)\,,$$
and
$$\underline H^{\rm sc}_{C_3}(a,m;\kappa)
=\underline H_{C_3}(a,m;\kappa)=\underline H^{\rm loc}_{C_3}(a,m;\kappa)\,.$$
\end{corol}

\begin{rem}\label{rem-int-1}
In \cite{Kach}, the second author has established the following
asymptotic expansion of $\underline H_{C_3}^{\rm loc}(a,m;\kappa)$
(it can also be shown to hold for $\overline H_{C_3}^{\rm
loc}(a,m;\kappa)$):
\begin{equation}\label{kach}
\underline H_{C_3}^{\rm
loc}(a,m;\kappa)=\frac{\kappa}{\alpha_0(a,m)}\left(1+o(1)\right)\quad{\rm
as~}\kappa\to\infty\,.
\end{equation}
The constant $\alpha_0(a,m)$ in \eqref{kach} satisfies
$\frac12<\alpha_0(a,m)<1$ for $m>1$ and $\alpha_0(a,m)=1$ otherwise,
and is defined via an auxiliary spectral problem (see
Theorem~\ref{kach-AsymptAnal} below).
\end{rem}

\begin{rem}\label{rem-m<1}
Theorem~\ref{thm-main} does not cover the regime $m\leq1$. In this
specific regime, we have by Remark~\ref{rem-int-1} that the
nucleation field $H_{C_3}$ coincides with the second critical field
$H_{C_2}$. This reflects one feature of superconductors surrounded
by normal materials, that surface superconductivity can be absent
(see \cite{deGeHu}).
However, we need specific tools to treat this interesting case,
which are beyond the scope of the present paper.
\end{rem}

We say that the transition from the superconducting to the normal
phase is sharp if we have equality between upper and lower fields. By
Corollary~\ref{thm-main-corol} it suffices to verify this for the
`local fields', i.e. look whether the following equality holds: 
$$\underline H_{C_3}^{\rm loc}(a,m;\kappa)=\overline H_{C_3}^{\rm
loc}(a,m;\kappa)\,?$$ One part of the present paper is devoted to
this question, which links to a monotonicity problem of magnetic
Schr\"odinger operators.

It is shown in \cite{Kach} that there exists $m_*\geq 1$ (see
\eqref{m*} below for a precise definition of $m_*$) such that for
$m> m_*$, the second correction term in (\ref{kach}) is of order $1$
and determined by the maximal value of the scalar curvature of
$\Omega$. The precise result in this case is the following:
\begin{equation}\label{HC3-curv}
\underline H_{C_3}(a, m;\kappa)=\frac{\kappa}{\alpha_0\left(a,m\right)}
+\frac{\mathcal
C_1\left(a,m\right)}{\alpha_0\left(a,m\right)^{3/2}}(\kappa_{\rm
r})_{\max}+\mathcal O\left(\kappa^{-1/3}\right),\quad \text{\rm as
}\kappa\to+\infty,\end{equation}
where the function
$\mathcal C_1(a,\cdot):\,[ m_*,+\infty[\,\mapsto\mathbb R_+$ is
    defined via an auxiliary self-adjoint operator (see (\ref{C1})) and
$\kappa_{\rm r}$ denotes the scalar curvature of $\partial\Omega$.

Our final result is that all the critical fields coincide for large
$\kappa$ under the extra condition that $m$ is sufficiently large.

\begin{thm}\label{thm:identification}
Given $\Omega, \widetilde \Omega$ and $a>0$, there exists $
m_0>0$ and if $m>m_0$, there exists $\kappa_0>0$ such that if 
$\kappa>\kappa_0$ then all the six
critical fields defined in \eqref{hc3-sc}, \eqref{hc3-gl} and
\eqref{hc3-loc} coincide. Furthermore, their common value is the
unique solution $H=H_{C_3}(a,m;\kappa)$ of the equation
\begin{align*}
  \mu^{(1)}(\kappa,H) = 0\,.
\end{align*}
In particular, the asymptotics \eqref{HC3-curv} holds for all the six
different definitions of the critical field. 
\end{thm}

As we mentioned earlier, the essential key for establishing
Theorem~\ref{thm:identification} is a monotonicity result concerning the first
eigenvalue of a Schr\"odinger type operator. Actually, let
$\mu_1(B,\alpha)$  be the first eigenvalue of the following operator
\begin{equation}\label{kach2}
-\nabla_{B\mathbf F}\cdot w_m(x)\nabla_{B\mathbf F}+\alpha
BV_a(x)\quad {\rm in~}\widetilde \Omega\,,
\end{equation}
with
\begin{equation}\label{w-V}
w_m(x)=\left\{\begin{array}{cl}
1&{\rm in ~}\Omega\\
\frac1m&{\rm in~}\mathbb R^2
\setminus\Omega\,,\end{array}\right.\quad V_a(x)=\left\{
\begin{array}{cl}
-1&{\rm in}~\Omega\\
a&{\rm in}~\mathbb R^2\setminus\Omega\,.
\end{array}
\right.\end{equation} Roughly speaking, we will
establish that the equation in $(B,\alpha)$
$$\mu_1(B,\alpha)=0$$
admits a unique solution provided that $B$ is large enough and
$\alpha$ remains close to $\alpha_0(a,m)$.

In order to prove Theorem~\ref{thm-main}, we establish the following
crucial estimate, valid for any critical point of the functional
$\mathcal G_{\kappa,H}$ (i.e. solution of \eqref{GL-eq} below),
$$a\int_{\widetilde\Omega\setminus\Omega}|\psi|^4\,dx\leq
\int_\Omega|\psi|^4\,dx\,.$$ The above estimate is non-trivial and
we prove it through a detailed analysis of the regularity of
critical points of the functional (\ref{GL-func}) (see
Lemma~\ref{reg-H2} and Theorem~\ref{partition-L4} below). With the
above estimate on the one  hand,  and other weak decay estimates
established in  Lemmas~\ref{standard} and \ref{lem-BoFo} on the
other hand, we prove Theorem~\ref{loc=glo}, which links the
existence of non-trivial critical points to a spectral condition on
the eigenvalue (\ref{EV}).   Theorem~\ref{thm-main} is then   a
consequence of Theorem~\ref{loc=glo}.

The paper is organized in the following way.
In Section~\ref{Sec:loc=glob}, we establish a necessary and
sufficient condition for the functional $\mathcal G_{\kappa,H}$ to
admit non-trivial critical points, see Theorem~\ref{loc=glo}.

In
Section~\ref{Sec:Monot} we establish a monotonicity result for the
eigenvalue $\mu_1(B,\alpha)$, see Theorem~\ref{thm-monton}. As a
consequence, we obtain equality of local fields stated in
Theorem~\ref{thm:identification}, 
Theorem~\ref{corol-monton} and Remark~\ref{rem:Conclusion}.

In Section~\ref{decay}, we discuss the decay of energy minimizing order
parameters, and we prove that they decay exponentially away from
the boundary provided that the magnetic field is sufficiently large.\\
Finally, in the appendix, we prove an improved
expansion for (\ref{HC3-curv}) in the particular case when $\Omega$ is
a disc domain.

\section{Proof of Theorem~\ref{thm-main}}\label{Sec:loc=glob}
\subsection{Basic estimates on solutions}
We return now to the analysis of critical points of the functional
(\ref{GL-func}). As we mentioned in the introduction, minimizers of
\eqref{GL-func}
exist in the space $H^1(\widetilde\Omega;\mathbb C)\times H^1
(\widetilde\Omega;\mathbb R^2)$ and they are
{\it weak} solutions of the associated
Ginzburg-Landau equations:
\begin{equation}\label{GL-eq}
\left\{
\begin{array}{l}
-\nabla_{\kappa H\Ab}\cdot w_m\nabla_{\kappa
H\Ab}\psi+\kappa^2\left(V_a\psi+
\mathbf 1_{\Omega}|\psi|^2\psi\right)=0\,,\\
\nabla^\bot {\rm curl}\,\Ab=(\kappa H)^{-1}w_m {\rm Im}\left(\overline{\psi}\,(\nabla-i\kappa
H\Ab)\psi\right)\,,\quad{\rm in~}\widetilde\Omega\,,\\
\nu\cdot(\nabla-i\kappa H\Ab)\psi=0\,,\quad{\rm curl}\,\Ab=1\quad{\rm
  on}~\partial\widetilde\Omega\,,
\end{array}
\right.
\end{equation}
where $\nu$ is the unit outward normal
vector of $\partial\widetilde\Omega$.

In order to avoid the question of regularity concerning solutions of
the above equation, we shall invoke energy arguments and use only
the weak formulation of (\ref{GL-eq}). More precisely,
$(\psi,\Aa)\in H^1(\widetilde\Omega;\mathbb C)\times
H^1(\widetilde\Omega;\mathbb R^2)$ is a weak solution of
(\ref{GL-eq}) if for all $\phi\in H^1(\widetilde\Omega;\mathbb C)$
and $\mathbf a\in H^1(\widetilde\Omega;\mathbb R^2)$, the following
equalities hold,
\begin{multline}
  \label{eq:15}
\int_{\widetilde\Omega}
\Big(w_m(x)(\nabla-i\kappa H\Aa)\psi\cdot\overline{(\nabla-i\kappa H
\Aa)\phi}\\
+\kappa^2(V_a(x)\psi+1_\Omega(x)|\psi|^2
\psi)\overline\phi\Big)\, d  x=0\,,
\end{multline}
\begin{align}
  \label{eq:16}
\int_{\widetilde\Omega}{\rm curl}\,(\Aa-\Fb)({\rm curl}\,\mathbf
a)-w_m(x)(\kappa H)^{-1}{\rm Im} \big(\overline{\psi}\,(\nabla-i\kappa H\Aa)\psi\big)\cdot
\mathbf a\, d  x=0\,.
\end{align}

A standard choice of gauge permits us to minimize (\ref{GL-func})
in the reduced space
$H^1(\widetilde\Omega;\mathbb C)\times H^1_{\rm
  div}(\widetilde\Omega;\mathbb R^2)$\,,
where
\begin{equation}\label{gauge}
 H^1_{\rm
  div}(\widetilde\Omega;\mathbb R^2)=
\{\Ab\in H^1(\widetilde\Omega;\mathbb R^2)~:~
{\rm div}\,\Ab=0\quad{\rm in}~\widetilde\Omega\,,\quad
\nu\cdot\nabla \Ab=0\quad{\rm on}~\partial\widetilde\Omega\}\,.
\end{equation}

We get as  immediate consequence of the Poincar\'e Lemma (see
e.g. \cite[p.~16, Theorem~1.5]{Te}):

\begin{lem}\label{lem-poincare}
There exists a constant $C>0$ such that for all $\Ab \in H^1_{\rm
  div}(\widetilde\Omega;\mathbb R^2)$,
\begin{equation}\label{poincare}
\|\Ab-\mathbf F\|_{H^1(\widetilde\Omega)}\leq C\|{\rm
  curl}\,\Ab-1\|_{L^2(\widetilde\Omega)}\,.
\end{equation}
Here $\mathbf F$ is the vector field introduced in (\ref{F}).
\end{lem}


\begin{lem}\label{standard}
Let $(\psi,\Ab)$ be a critical configuration of (\ref{GL-func}),
i.e. a weak solution of \eqref{GL-eq}. Then the following
estimates hold~:
\begin{equation}\label{Linfty}
\|\psi\|_{L^\infty(\widetilde\Omega)}\leq1\,,
\end{equation}
\begin{equation}\label{nabla}
\|(\nabla-i\kappa H\Ab)\psi\|_{L^2(\widetilde\Omega)}\leq
\max(1,\sqrt{m})\kappa\|\psi\|_{L^2(\Omega)}\,,
\end{equation}
\begin{equation}\label{partition}
a\int_{\widetilde\Omega\setminus\Omega}|\psi|^2\, d  x\leq \int_\Omega |\psi|^2\, d  x\,,
\end{equation}
\begin{equation}\label{L2}
\|\psi\|_{L^4(\Omega)}^2\leq
\|\psi\|_{L^2(\Omega)}\,,
\end{equation}and
\begin{equation}\label{curl}
H\|{\rm curl}\,\Ab-1\|_{L^2(\widetilde\Omega)}\leq
C\|\psi\|_{L^2(\Omega)}\|\psi\|_{L^4(\widetilde\Omega)}\,.
\end{equation}
Here the constant $C>0$ depends only on $a$, $m$ and
$\widetilde\Omega$.
\end{lem}

\begin{proof}
The estimate \eqref{Linfty} is rather standard and is obtained in
\cite{Gi}. It can be derived
using a simple energy argument as in \cite{DGP}, without relying on
regularity properties of $(\psi,\Aa)$.

Inserting $\phi = \psi$ in \eqref{eq:15}
we get,
\begin{equation}\label{c-pts}
\int_{\widetilde\Omega}\left(
w_m(x)|(\nabla-i\kappa H\Ab)\psi|^2+\kappa^2V_a(x)|\psi|^2
+\kappa^2 1_\Omega(x)|\psi|^4\right)\, d  x= 0\,.
\end{equation}
Now \eqref{nabla} and \eqref{partition}
are consequences of \eqref{c-pts}.
The estimate \eqref{L2} is a consequence of \eqref{Linfty}.

Now, we prove \eqref{curl}. Up to a gauge transformation, we may
 assume that $\Ab\in H^1_{\rm div}(\psi,\Ab)$.
Inserting $\mathbf a = \Ab - \Fb$ in \eqref{eq:16} and estimating we
get,
\begin{eqnarray*}
&&\hskip-1cm\kappa H\int_{\widetilde\Omega}|{\rm curl}(\Ab-\mathbf F)|^2\, d  x\\
&&\hskip0.5cm\leq \max\left(1,\frac1m\right)
\|\psi\|_{L^\infty(\widetilde\Omega)}\|(\nabla-i\kappa
H\Ab)\psi\|_{L^2(\widetilde\Omega)}\|(\Ab-\mathbf
F)\psi\|_{L^2(\widetilde\Omega)}.
\end{eqnarray*}
Invoking (\ref{Linfty}),  (\ref{nabla}) and applying a further
Cauchy-Schwarz inequality, we get,
\begin{multline}
  \label{eq:17}
  \kappa H\int_{\widetilde\Omega}|{\rm curl}(\Ab-\mathbf F)|^2\, d 
  x
\leq\\
\kappa\max\left(1,\frac1{\sqrt{m}}\right)\|\psi\|_{L^2(\Omega)}
\|\psi\|_{L^4(\widetilde\Omega)}\|\Ab-\mathbf
F\|_{L^4(\widetilde\Omega)}.
\end{multline}
%
Now, by the Sobolev embedding theorem and
Lemma~\ref{lem-poincare}, we get
 $$\|\mathbf A-\mathbf F\|_{L^4(\widetilde\Omega)}\leq C_{\rm
 Sob}\|\mathbf A-\mathbf F\|_{H^1(\Omega)}\leq \widetilde C\|{\rm curl}(\mathbf A-\mathbf
 F)\|_{L^2(\widetilde\Omega)}.$$
Thus, we can divide through by $\| \curl \Ab - \Fb
\|_{L^2(\widetilde{\Omega})}$ in \eqref{eq:17} to get \eqref{curl}.
\end{proof}

The next theorem gives the finiteness of the critical fields $H_{C_3}$
and is well-known (see
\cite{GiPh,FoHe07} for superconductors in vacuum and \cite{Gi} for a
setting similar to ours). We give an easy (spectral) proof for completeness.

\begin{thm}\label{thm-GP-type}
There exists a constant $C>0$ such that if
\begin{align} \label{eq:20}
  \kappa\geq 1\quad{\rm and} \quad H \geq C \kappa,
\end{align}
then the only weak solution to \eqref{GL-eq} is the normal state $(0, \Fb)$.
\end{thm}

\begin{proof}
We may assume---after possibly performing a gauge
transformation---that the stationary point satisfies $(\psi, \Ab) \in
H^1(\widetilde\Omega;\mathbb C)\times H^1_{\rm
  div}(\widetilde\Omega;\mathbb R^2)$.
Suppose that $\psi \neq 0$.

Using $\| \psi \|_{\infty} \leq 1$, we have the pointwise inequality
\begin{align}
  \label{eq:22}
  |(\nabla-i\kappa H\Fb)\psi|^2 \leq 2 |(\nabla-i\kappa H\Ab)\psi|^2 +
  2(\kappa H)^2 |\Ab - \Fb|^2.
\end{align}
Upon integration of \eqref{eq:22} and using Lemma~\ref{standard}, the
Sobolev inclusion $H^1 \hookrightarrow L^4$ and
Lemma~\ref{lem-poincare}, we find
\begin{align}
  \label{eq:23}
  \int_{\widetilde \Omega} |(\nabla-i\kappa H\Fb)\psi|^2\,dx \leq
  C \kappa^2 \| \psi \|_{L^2(\widetilde \Omega)}^2,
\end{align}
for some constant $C>0$.
This implies---since $\psi \neq 0$ by assumption---that the lowest Neumann eigenvalue $\mu^N(\kappa H)$ of
$-(\nabla-i\kappa H\Fb)^2$ in $L^2(\widetilde \Omega)$ satisfies
\begin{align}
  \label{eq:24}
  \mu^N(\kappa H) \leq C \kappa^2.
\end{align}
However, since $\widetilde \Omega$ is smooth, we have (see \cite{FoHe07})
\begin{align}
  \label{eq:25}
  \mu^N(\kappa H) \geq C'  \kappa H\,,\quad\forall~H\geq H_0,
\end{align}
for positive constants $C'$ and $H_0$ independent from $\kappa$.
Combining \eqref{eq:25} and \eqref{eq:24} yields the result.
\end{proof}

In the next theorem, we show that the $L^4$ analogue of
\eqref{partition} holds for critical points.

\begin{thm}\label{partition-L4}
Let $(\psi,\Aa)\in H^1(\widetilde\Omega;\mathbb C)\times H^1_{\rm
div}(\widetilde\Omega;\mathbb R^2)$ be a weak solution of (\ref{GL-eq}). Then the
following estimate holds,
$$a\int_{\widetilde\Omega\setminus\Omega}|\psi|^4\, d  x\leq
\int_{\Omega}|\psi|^4\, d  x\,.$$
\end{thm}

The estimate in Theorem~\ref{partition-L4} is an essential
ingredient in proving Theorem~\ref{loc=glo} below (which links the
existence of non-trivial critical points to a spectral condition).
In order to prove Theorem~\ref{partition-L4}, we need to establish
some regularity properties for solutions of (\ref{GL-eq}).

\begin{lem}\label{reg-H2}
Let $(\psi,\Aa)\in H^1(\widetilde\Omega;\mathbb C)\times H^1_{\rm
div}(\widetilde\Omega;\mathbb R^2)$ be a weak solution of
(\ref{GL-eq}). Then,
\begin{enumerate}
\item $(\psi,\Aa)\in H^2(\widetilde\Omega\setminus\partial\Omega\,;\mathbb C)\times
H^2(\widetilde \Omega\,;\mathbb R^2)$\,.
\item $(\psi,\Aa)\in C^\infty(\widetilde\Omega\setminus\partial\Omega\,;\mathbb C)\times
C^\infty(\widetilde\Omega\setminus\partial\Omega\,;\mathbb R^2)$.
\end{enumerate}
\end{lem}

Let us emphasize that the complete regularity of $(\psi,\Aa)$ up to
the boundary of $\Omega$ does not follow from standard regularity
theory for elliptic PDE, and hence deserves to be studied
independently.

\begin{proof}[Proof of Lemma~\ref{reg-H2}.]
The interior regularity of $(\psi,\Aa)$ (statement (2) above) is
obtained through a
standard bootstrapping argument, see \cite[Proposition~3.8]{SaSe}.

We move now to the $H^2$ regularity. Since $\Aa\in H^1_{\rm div}
(\widetilde\Omega)$, the equation for $\Aa$ in (\ref{GL-eq}) becomes,
$$-\Delta \Aa=g(x)\quad{\rm in}~\widetilde\Omega\,,$$
coupled with the boundary conditions,
$${\rm curl}\,\Aa=1\,,\quad\nu\cdot\Aa=0\,,
\quad{\rm on}~\partial\widetilde\Omega\,.$$
Here
$$g(x):=(\kappa H)^{-1}w_m(x)\left(i\psi,(\nabla-i\kappa
H\Aa)\psi\right)\in L^2(\widetilde\Omega)\,.
$$
Now $\Aa\in H^2(\widetilde\Omega)$ follows from the
$W^{k,p}$-regularity of the ${\rm curl}$-${\rm div}$ system (see
\cite{ADN}).

To obtain the regularity of $\psi$ one has to be more careful, since
$\psi$ satisfies formally a transmission condition on
$\partial\Omega$,
\begin{equation}\label{eq-interface}
\mathcal T_{\partial\Omega}^{\rm int}(\nu\cdot(\nabla -i\kappa
H\Aa)\psi)=\frac1m \mathcal T_{\partial\Omega}^{\rm
ext}(\nu\cdot(\nabla-i\kappa H\Aa)\psi)\,.
\end{equation}
Here
\begin{equation}\label{eq-trace}
\mathcal T_{\partial\Omega}^{\rm int}:H^1(\Omega)\to
L^2(\partial\Omega)\,,\quad \mathcal T_{\partial\Omega}^{\rm
ext}:H^1(\mathbb R^2\setminus\Omega)\to
L^2(\partial\Omega)\end{equation} are respectively, the `interior'
and `exterior' trace operators, and $\nu$ is the outward unit normal
vector of $\partial\Omega$.

The crucial point is now to apply a suitable gauge transformation
which makes the transmission condition (\ref{eq-interface})
independent of the vector potential $\Aa$.

Let $\chi$ be the solution of the following boundary value problem,
$$
-\Delta\chi=0\quad{\rm in}~\Omega\,,\quad
\nu\cdot\nabla\chi=\nu\cdot \Ab\quad{\rm on}~\partial\Omega\,.$$
Since $\Ab\in H^2(\widetilde\Omega)$, standard regularity theory
gives $\chi\in W^{2,p}(\Omega)$ for every $p\geq1$. Let us fix a
choice of $p>2$ such that Sobolev embedding gives $W^{2,p}\subset
C^{1,\alpha}$ for some $\alpha\in]0,1[$.

Given a smooth domain $K$ such that $\overline\Omega\subset K\subset
\overline K\subset\widetilde\Omega$, we can extend $\chi$ to a
function $\widetilde\chi\in W^{2,p}(\widetilde\Omega)$ such that
${\rm supp}\,\chi\subset K$ (see Remark on p.257 in \cite{Ev}).

Now, defining,
$$\varphi=\psi e^{i\widetilde\chi}\,,\quad\Bb= \Ab -\nabla\widetilde\chi\,,$$
the condition (\ref{eq-interface}) reads formally,
\begin{equation}\label{eq-interface1}
\mathcal T_{\partial\Omega}^{\rm int}(\nu\cdot\nabla\varphi)=
\frac1m\mathcal T_{\partial\Omega}^{\rm
ext}(\nu\cdot\nabla\varphi)\,.
\end{equation}
Actually, the equation for $\varphi$ becomes,
$$-{\rm div}\left(w_m(x)\nabla \varphi\right)=f\quad{\rm in~}\widetilde\Omega\,,$$
with
\begin{multline*}
  f(x)=-(\kappa H)w_m(x)\left[2i(\Bb\cdot\nabla)+i({\rm
    div}\,\Bb)
+\kappa H|\Bb|^2\right]\varphi\\
-\kappa^2\left(V_a(x)\varphi+\mathbf
    1_\Omega(x)
|\varphi|^2\varphi\right)\,.
\end{multline*}
Here $w_m$ and $V_a$ are as in (\ref{w-V}). Moreover, the equation
for $\varphi$   is supplemented with the boundary condition,
$$\nu\cdot\nabla \varphi=0\quad{\rm on~}\partial\widetilde\Omega\,.$$
The obtained equation for $\varphi$ is of the form studied in
\cite[Appendix~B]{Kach2}, where $L^2$-type estimates are shown to hold for
the solutions. Let us see how we will implement the aforementioned point.

Lemma~\ref{standard} gives $|\varphi|\leq
    1$. Moreover, since
$\varphi\in H^1(\widetilde\Omega)$, $\Ab\in H^2(\widetilde\Omega)$
and $\Bb=\Ab+\nabla\widetilde\chi$ is bounded, we deduce that $f\in
    L^2(\widetilde\Omega)$. Applying now Theorem~B.1 in
    \cite{Kach2}, we deduce that $\varphi\in
    H^2(\widetilde\Omega\setminus\partial\Omega)$. Since $\psi=\varphi e^{-i\widetilde\chi}$
    and $\widetilde\chi\in W^{2,p}(\widetilde\Omega)\subset C^{1,\alpha}\big(\,\overline{\widetilde\Omega}\,\big)$, then
    $\psi\in H^2(\widetilde\Omega\setminus\partial\Omega)$.
\end{proof}

\begin{proof}[Proof of Theorem~\ref{partition-L4}]
Thanks to Lemma~\ref{reg-H2}, $(\psi,\Aa)\in
H^1(\widetilde\Omega;\mathbb C)\times H^1_{\rm
div}(\widetilde\Omega;\mathbb R^2)$ being a solution of
(\ref{GL-eq}), the function $\psi\in
C^\infty(\widetilde\Omega\setminus\partial\Omega)$, and $(\psi,\Aa)$
is a strong solution of
(\ref{GL-eq}) in $\widetilde\Omega\setminus\partial\Omega$.

Consequently, the function,
$$u=|\psi|^2=\psi\,\overline\psi\in C^\infty(\widetilde\Omega\setminus\partial\Omega\,;\mathbb
R).$$ It is easy to verify that,
$$\frac12\Delta|\psi|^2={\rm Re}\big( \overline\psi(\nabla-i\kappa H\Aa)^2\psi\big)+|(\nabla-i\kappa H\Aa)\psi|^2\,,\quad\forall~x
\in\widetilde\Omega\setminus\partial\Omega\,.$$ Thus, defining the
function,
$$f(x)=\frac{1}{2} w_m(x)|(\nabla-i\kappa H\Aa)\psi(x)|^2,$$
we get that $u=|\psi|^2$ is a strong solution in
$\widetilde\Omega\setminus\partial\Omega$ of the equation,
\begin{equation}\label{eq-|psi|^2}
-\frac12{\rm div}\left(w_m\nabla u\right)+\kappa^2(V_a +\mathbf 1_\Omega u)u+f = 0\,.
\end{equation}
Here we remind the reader that the functions $w_m$ and $V_a$ are
defined in (\ref{w-V}).

Let us show that $u$ is a weak solution of (\ref{eq-|psi|^2}) in
$\widetilde\Omega$. To that end, we need only  verify that
\begin{equation}\label{transmission}
\mathcal T_{\partial\Omega}^{\rm int}(\nu\cdot\nabla u)=\frac1m
\mathcal T_{\partial\Omega}^{\rm ext}(\nu\cdot\nabla u)\quad{\rm
in}~L^2(\partial\Omega)\,,\end{equation} where $\mathcal
T_{\partial\Omega}^{\rm int}$ and $\mathcal T_{\partial\Omega}^{\rm
ext}$ are the trace operators introduced in (\ref{eq-trace}), and
$\nu$ is the outward unit normal vector of $\partial\Omega$.

Notice that $\nabla u=2\,{\rm Re}\,(\psi\nabla\overline\psi)\in
H^1(\widetilde\Omega\setminus\partial\Omega)$ since $\psi\in
H^2(\widetilde\Omega\setminus\partial\Omega;\mathbb C)$ and
$\widetilde\Omega\subset\mathbb R^2$, hence the trace of
$\nu\cdot\nabla u$ is
well defined in the usual sense.

Now, it is easy to verify that,
$$\nabla u =2\,{\rm Re}\,(\psi\nabla\overline\psi)=2\,{\rm
Re}\,(\psi\overline{(\nabla-i\kappa H\Aa)\psi}).$$ On the other
hand, since $(\psi,\Aa)\in
H^2(\widetilde\Omega\setminus\partial\Omega;\mathbb C)\times
H^2(\widetilde\Omega;\mathbb R^2)$ and weak solutions of
(\ref{GL-eq}), they satisfy in particular,
$$\mathcal T_{\partial\Omega}^{\rm int}(\nu\cdot(\nabla -i\kappa H\Aa)\psi)=\frac1m \mathcal
T_{\partial\Omega}^{\rm ext}(\nu\cdot(\nabla-i\kappa H\Aa)\psi)$$ in
$L^2(\partial\Omega)$. This shows that (\ref{transmission}) also
holds in $L^2(\partial\Omega)$ and consequently $u=|\psi|^2$ is a
weak solution in $\widetilde\Omega$ of (\ref{eq-|psi|^2}).

Therefore, multiplying (\ref{eq-|psi|^2}) by $u$ and integrating, we
get,
$$\int_{\widetilde\Omega}\left(\frac12w_m(x)|\nabla
u|^2+V_a(x)u^2+1_\Omega(x)u^3+f(x)u\right)\, d  x=0\,,
$$
with $u\geq0$ and $f\geq 0$. This yields the  estimate of
Lemma~\ref{partition-L4}.\end{proof}

\subsection{Weak Decay Estimate and applications}

Just as in \cite{BoFo}, we can derive the following weak decay
estimate.

\begin{lem}\label{lem-BoFo}
Assume that $a>0$ and $m>0$.
There exist positive constants $C$ and $C'$ such that,
if $(\psi,\Ab)$ is a solution of (\ref{GL-eq}) with:
$$\kappa(H-\kappa)\geq C\,,$$
then the following estimate holds:
\begin{align*}
\max\left(1,\frac{a\kappa}{H-\kappa}\right)
\|\psi\|_{L^2(\widetilde\Omega)}^2&\leq
C\int_{\{x\in\widetilde\Omega:\sqrt{\kappa(H-\kappa)}\,{\rm
    dist}(x,\partial\Omega)\leq1\}}|\psi(x)|^2\, d  x\\
&\leq
\frac{C'}{\kappa(H-\kappa)}\,.
\end{align*}
\end{lem}

\begin{proof}
The last inequality is an easy consequence of (\ref{Linfty}), so we
will only establish the first one.

 Let $\chi\in C^\infty(\mathbb
R)$ be a standard cut-off function such that,
$$\chi=1\quad{\rm in}~[1,\infty[\,,\quad
\chi=0\quad{\rm in}~]-\infty,1/2]\,.$$
Define
$\lambda=1/\sqrt{\kappa(\kappa-H)}$, and
$$\chi_\lambda(x)=\chi\left(\frac{{\rm
      dist}(x,\partial\Omega)}{\lambda}\right)\,,\quad
x\in\widetilde\Omega\,.$$
A simple calculation yields the following localization formula:
\begin{multline*}
\int_{\widetilde\Omega}w_m(x)\left(|(\nabla-i\kappa
  H\Ab)\chi_{\lambda}\psi|^2-|\nabla\chi_\lambda|^2|\psi|^2\right)\, d  x\\
={\rm Re}
\int_{\widetilde\Omega}\left(w_m(x)\overline{(\nabla-i\kappa H\Ab)\psi}\cdot
(\nabla-i\kappa H\Ab)(\chi_\lambda^2\psi)\right)\, d  x\,,
\end{multline*}
where $w_m$ is introduced in \eqref{w-V}.

Now, we use  the weak formulation of the first G-L equation in
(\ref{GL-eq}) and  get,
\begin{multline}\label{weak-loc}
\int_{\widetilde\Omega}w_m(x)\left(|(\nabla-i\kappa
  H\Ab)\chi_{\lambda}
  \psi|^2-|\nabla\chi_\lambda|^2|\psi|^2\right)\, d  x
+a\kappa^2\int_{\widetilde\Omega\setminus\Omega}\chi_\lambda^2|\psi|^2\, d 
x\\
=\kappa^2\int_\Omega\chi_\lambda^2(1-|\psi|^2)|\psi|^2\, d 
x\,.
\end{multline}
The next step is to give a lower bound to the first term on the left
hand side of \eqref{weak-loc}. This is done via an elementary
inequality from the spectral theory of magnetic Schr\"odinger
operators (see e.g. \cite[Thm 2.9]{AHS} or \cite[Lemma~2.4.1]{FoHe07}).
 Actually, since the function 
$\chi_\lambda$ has compact support not meeting
the boundary $\partial\Omega$, and $w_m=1$ in $\Omega$, the following
inequality holds,
$$\int_{\widetilde\Omega}w_m(x)|(\nabla-i\kappa H\Ab)\chi_\lambda
\psi|^2\, d  x\geq
\kappa H\int_{\Omega}({\rm
  curl}\,\Ab)|\chi_\lambda\psi|^2\, d  x\,.$$
Writing
$\curl \Ab=1+(\curl\Ab-1)$ then applying
a Cauchy-Schwarz inequality, we get,
\begin{eqnarray*}
\int_{\widetilde\Omega}w_m(x)|(\nabla-i\kappa H\Ab) \chi_{\lambda} \psi|^2\, d 
x&\geq&
\kappa H\int_{\Omega}w_m(x)|\chi_\lambda\psi|^2\, d  x\\
&&-\kappa H\|{\rm
  curl}\,\Ab-1\|_{L^2(\Omega)}\|\chi_\lambda\psi\|_{L^4(
\Omega)}^2\,.
\end{eqnarray*}
Implementing the estimates \eqref{curl}, \eqref{L2} and
\eqref{partition} we get, for some constant $c>0$, 
the following lower bound,
\begin{multline}\label{trick-weak-loc}
\int_{\widetilde\Omega}w_m(x)|(\nabla-i\kappa H\Ab)\chi_{\lambda}
\psi|^2\, d 
x\\
\geq\kappa H\int_{\Omega}|\chi_\lambda\psi|^2\, d  x
-c\kappa\|\psi\|_{L^2(\Omega)}
\|\chi_\lambda\psi\|_{L^4(\Omega)}^2\,.
\end{multline}
Upon substitution in (\ref{weak-loc}) and a
rearrangement of terms, we deduce,
\begin{multline*}
\kappa(H-\kappa)\int_\Omega|\chi_\lambda\psi|^2\, d 
x+a\kappa^2\int_{\widetilde\Omega\setminus\Omega}
|\chi_\lambda\psi|^2\, d  x
\leq c\kappa
\|\psi\|_{L^2(\Omega)}
\|\chi_\lambda\psi\|_{L^4(\Omega)}^2\\
+\|\chi'\|^2_{L^\infty(\mathbb
  R)}\lambda^{-2}\int_{\{{\rm dist}(x,\partial\Omega)\leq
  \lambda\}}w_m(x)|\psi|^2\, d 
x-\kappa^2\int_{\Omega}\chi_\lambda ^2|\psi|^4\, d 
x\,.
\end{multline*}
Implementing again a Cauchy-Schwarz inequality, we get
\begin{eqnarray*}
&&\kappa(H-\kappa)\int_\Omega|\chi_\lambda\psi|^2\, d 
x+a\kappa^2\int_{\widetilde\Omega\setminus\Omega}
|\chi_\lambda\psi|^2\, d  x\\
&&\hskip0.5cm
\leq c^2\|\psi\|^2_{L^2(\Omega)}+
\|\chi'\|^2_{L^\infty(\mathbb
  R)}\lambda^{-2}\int_{\{{\rm dist}(x,\partial\Omega)\leq
  \lambda\}}w_m(x)|\psi|^2\, d 
x\\
&&\hskip0.5cm+\kappa^2\int_{\Omega}(\chi_\lambda^4-\chi_\lambda^2)|\psi|^4\, d 
x
\,,
\end{eqnarray*}
where the last term on the right hand side above is negative, since
$0\leq\chi_\lambda\leq 1$.

Decomposing the integral
$\displaystyle\int_{\Omega}|\psi|^2=
\displaystyle\int_{\Omega}|\chi_\lambda\psi|^2+\displaystyle
\int_{\Omega}(1-\chi_\lambda^2)|\psi|^2$,
and assuming that
$$\kappa(H-\kappa)\geq 2c^2\,,$$ we get
\begin{eqnarray*}
&&\frac1{2}\max\left(
\kappa(H-\kappa),a\kappa^2\right)
\int_{\widetilde\Omega}|\chi_\lambda\psi|^2\, d 
x\\
&&\hskip0.5cm\leq \left(c^2
+\max\left(1,\frac1m\right)\|\chi'\|^2_{L^\infty(\mathbb
    R)}\lambda^{-2}\right)\int_{\{{\rm dist}(x,\partial\Omega)\leq
  \lambda\}}|\psi|^2\, d  x\,.
\end{eqnarray*}
Recall that $\lambda=1/\sqrt{\kappa(H-\kappa)}\,$. The conditions
on $\chi$ and $\kappa(H-\kappa)$ imply that
$$ 
\max\left(1,\frac{a\kappa}{H-\kappa}\right)
\int_{\widetilde\Omega}|\chi_\lambda\psi|^2\, d 
x\leq 4\max\left(1,\frac1m\right)\|\chi'\|^2_{L^\infty(\mathbb R)}
\int_{\{{\rm dist}(x,\partial\Omega)\leq
  \lambda\}}|\psi|^2\, d  x\,.$$
Consequently, we get, 
\begin{multline*}
\max\left(1,\frac{a\kappa}{H-\kappa}\right)
\int_{\widetilde\Omega}|\psi|^2\, d  x\\\leq
\left( 4\max\left(1,\frac1m\right)\,\|\chi'\|^2_{L^\infty(\mathbb R)}+1\right)\int_{\{{\rm dist}(x,\partial\Omega)\leq
  \lambda\}}|\psi|^2\, d  x\,.
\end{multline*}
Choosing
$C=\max\left(c^2,4\max\left(1,\frac1m\right)\,\|\chi'\|^2_{L^\infty(\mathbb
    R)}+1\right)$, we get the desired bound.
\end{proof}

The next theorem gives a purely spectral criterion for the existence
of non-trivial critical points of (\ref{GL-func}).

\begin{thm}\label{loc=glo}
Given $a>0$ and $m>1$, there exists $\kappa_0>0$ such that, for all $\kappa\geq\kappa_0$,
the following two statements are equivalent:
\begin{enumerate}
\item There exists a solution $(\psi,\Ab)$ of (\ref{GL-eq}) with
  $\|\psi\|_{L^2(\widetilde\Omega)}\not=0$\,.
\item The parameters $\kappa$ and $H$ satisfy $\mu^{(1)}(\kappa,H)<0$,
  where the eigenvalue $\mu^{(1)}(\kappa,H)$ is introduced in
  (\ref{EV}).
\end{enumerate}
\end{thm}

\begin{proof}
It is well known that the second statement implies the first one.
Actually, we only use $(t\psi_*,0)$, with $t$ sufficiently small and
$\psi_*$ an eigenfunction associated with $\mu^{(1)}(\kappa,H)$, as
a test configuration for the functional (\ref{GL-func}). The
resulting energy will be lower that that of a normal state. Hence a
minimizer, which is a solution of (\ref{GL-eq}),
will be non-trivial.

We assume now that the first statement holds and we show that the
second statement is true provided that $\kappa$ is sufficiently
large.

Thanks to Theorem~\ref{thm-GP-type}, it is sufficient to deal with
applied magnetic fields satisfying $H\leq C\kappa$, for some
constant $C>0$.

On the other hand, if $c>0$ is a sufficiently small constant, then
we can show that
\begin{equation}\label{eq-mu<0}
\mu^{(1)}(\kappa,H)<0\quad{\rm  for ~all~}\quad H<c\kappa\,.
\end{equation}
In order to see this, notice that the variational min-max principle
gives,
$$\mu^{(1)}(\kappa,H)\leq \mu^D(\kappa H;\Omega)-\kappa^2\,,$$
where $\mu^D(\kappa H;\Omega)$ is the lowest Dirichlet eigenvalue of
$- (\nabla-i\kappa H\mathbf F)^2$ in $L^2(\Omega)$. 
Standard estimates on Dirichlet realizations of magnetic operators
(see \cite{HeMo1}) yield the existence of a constant $\tilde C>0$ such
that $$\mu^D(B;\Omega)\leq \tilde C\max(B,1)\quad\forall~B>0\,,$$
implying \eqref{eq-mu<0}.

Therefore, we
restrict ourselves to applied magnetic fields satisfying,
$$c\kappa\leq H\leq C\kappa\,.$$

Using Lemma~\ref{lem-BoFo} and a Cauchy-Schwarz inequality,
we find a positive constant $C>0$ such
that
\begin{equation}\label{proof1}
\|\psi\|_{L^2(\widetilde\Omega)}^2\leq C\left(\int_{ \{{\rm
dist}(x,\partial\Omega)\leq\frac1\kappa\}} d  x\right)^{1/2}
\|\psi\|_{L^4(\widetilde \Omega)}^2    \leq \frac{C'}{\sqrt{\kappa}}
\|\psi\|_{L^4(\Omega)}^2\,,
\end{equation}
where we use Theorem~\ref{partition-L4} to get the last inequality.

Since $(\psi,\Ab)$ is a (weak) solution of \eqref{GL-eq}, we get by
setting $\phi = \psi$ in
\eqref{eq:15} together with the assumption on $\psi$ that
\begin{equation}\label{Delta>0}
0<\kappa^2\|\psi\|_{L^4(\Omega)}^4\leq
-\int_{\widetilde\Omega}\left(w_m(x)|(\nabla-i\kappa H\Ab)\psi|^2+
\kappa^2 V_a(x)|\psi|^2 \right) d  x =:\Delta\,,\end{equation} where
$w_m$ and $V_a$ are introduced in (\ref{w-V}). Notice that 
use \eqref{partition} implies that $\|\psi\|_{L^4(\Omega)}\not=0$.

Implementing \eqref{Delta>0} in \eqref{proof1}, we get
\begin{equation}\label{proof2}
\|\psi\|_{L^2(\widetilde\Omega)}^2\leq
C''\sqrt{\Delta}\,\kappa^{-3/2}\,.
\end{equation}
We estimate
\begin{eqnarray*}
&&|(\nabla-i\kappa H\mathbf \Ab)\psi|^2\\
&&\hskip0.5cm\geq(1-\sqrt{\Delta}\,\kappa^{-3/4})
|(\nabla-i\kappa H\mathbf
F)\psi|^2-\frac1{\sqrt{\Delta}\,\kappa^{-3/4}}(\kappa H)^2|(\Ab-\mathbf
F)\psi|^2\,.
\end{eqnarray*}
This yields the following estimate on $\Delta$:
\begin{eqnarray}\label{estimate-Delta}
&&\Delta\leq
  \left(
-\mu^{(1)}(\kappa,H)+\min\left(1,\frac1m\right)\sqrt{\Delta}\,\kappa^{-3/4}
\mu^N(\kappa
  H)\right)\|\psi\|_{L^2(\widetilde\Omega)}^2\\
&&\hskip0.8cm+\min\left(1,\frac1m\right)\frac{\kappa^{3/4}}{\sqrt{\Delta}}
(\kappa H)^2\int_{\widetilde\Omega}|(\Ab-\mathbf F)\psi|^2\, d  x\,,\nonumber
\end{eqnarray}
where $\mu^N(\kappa H)$ is the lowest eigenvalue of the magnetic
Neumann Laplacian\break $-(\nabla-i\kappa H\mathbf F)^2$ in
$\widetilde\Omega$.

We recall that $H=\mathcal O(\kappa)$.
Using the asymptotic behavior of $\mu^N(B)$ as $B \to\infty$ together
with the estimate (\ref{proof2}), we get:
\begin{equation}\label{proof3}
\sqrt{\Delta}\,\kappa^{-3/4}\mu^N(\kappa
  H)\|\psi\|_{L^2(\widetilde\Omega)}^2\leq C\kappa^{-1/4}\Delta\,.
\end{equation}
On the other hand, using a Cauchy-Schwarz inequality and the estimate
(\ref{Delta>0}), we get:
$$(\kappa H)^2\int_{\widetilde\Omega}|(\Ab-\mathbf F)\psi|^2\, d  x\leq
(\kappa H)^2\|\Ab-\mathbf F\|_{L^4(\widetilde\Omega)}^2
\frac{\sqrt{2\Delta}}{\kappa}\,.$$
By Sobolev embedding and Lemma~\ref{lem-poincare},
$$(\kappa H)^2
\|\Ab-\mathbf F\|_{L^4(\widetilde\Omega)}^2\leq C(\kappa H)^2 \|{\rm
curl}\,\Ab-1\|_{L^2(\widetilde\Omega)}^2\leq
C\kappa^2\|\psi\|_{L^4(\widetilde\Omega)}^4\,.$$ Using again
Theorem~\ref{partition-L4}, we deduce that,
$$(\kappa H)^2
\|\Ab-\mathbf F\|_{L^4(\widetilde\Omega)}^2\leq
C\kappa^2\|\psi\|_{L^4(\Omega)}^4\leq C\Delta\,.$$
 Therefore, implementing all
the above estimates in (\ref{estimate-Delta}), we get
$$\Delta\leq
-\mu^{(1)}(\kappa,H)\|\psi\|_{L^2(\widetilde\Omega)}^2+C\frac{\Delta}
{\kappa^{1/4}}\,.$$
Knowing that $\Delta>0$ (see (\ref{Delta>0})), we deduce for $\kappa$
sufficiently large
the desired inequality, $\mu^{(1)}(\kappa,H)<0$.
\end{proof}

\begin{proof}[Proof of Theorem~\ref{thm-main}]
We can now finish the proof of Theorem~\ref{thm-main} as follows.
Theorem~\ref{loc=glo} gives $\mathcal N^{\rm
sc}(a,m;\kappa)=\mathcal N^{\rm loc}(a,m;\kappa)$ for $\kappa$
sufficiently large. On the other hand, we have the trivial
inclusions $\mathcal N(a,m;\kappa)\subset\mathcal N^{\rm
  sc}(a,m;\kappa)$ and $\mathcal N^{\rm
  loc}(a,m;\kappa)\subset\mathcal N(a,m;\kappa)$.
\end{proof}

\section{Monotonicity of the first eigenvalue}\label{Sec:Monot}
\subsection{The constant $\alpha_0(a,m)$}
We recall in this section the definition of the constant
$\alpha_0(a,m)$ introduced in \cite{Kach}, together with the main
properties of a family of ordinary differential operators.

Consider the space  $B^k(\mathbb R)=H^k(\mathbb R)\cap L^2(\mathbb
R;\,|t|^k d  t)$, $k\in\mathbb N$. Given $a,m,\alpha>0$ and $\xi\in\mathbb R$, let us
define the quadratic form~:
\begin{equation}\label{II-qf-xi}
B^1(\mathbb R)\ni u\mapsto q[a,m,\alpha;\xi](u),
\end{equation}
where~:
\begin{eqnarray}\label{IIfq}
&&q[a,m,\alpha;\xi](u)=\int_{\mathbb
R_+}\left(|u'(t)|^2+|(t-\xi)u(t)|^2-\alpha|u(t)|^2\right)dt\\
&&\hskip3cm+\int_{\mathbb
R_-}\left(\frac1m\left[|u'(t)|^2+|(t-\xi)u(t)|^2\right]+a\alpha|u(t)|^2\right)dt.\nonumber
\end{eqnarray}
We denote by $H[a,m,\alpha;\xi]$ the self-adjoint operator
associated to the closed symmetric  quadratic form (\ref{II-qf-xi}).
The domain of $H[a,m,\alpha;\xi]$ is defined by~:
\begin{equation}\label{IIdOp}
D(H[a,m,\alpha;\xi])=\left\{u\in B^1(\mathbb R);\quad u_{|_{\mathbb
R_{\pm}}}\in B^2(\mathbb R_{\pm}),\quad u'(0_+)=\frac1m
u'(0_-)\right\},
\end{equation}
and for $u\in D(H[a,m,\alpha;\xi])$, we have,
\begin{equation}\label{II-H-}
\left(H[a,m,\alpha;\xi]u\right)(t)=\left\{\begin{array}{l}
\left[\left(-\partial_t^2+(t-\xi)^2-\alpha\right)u\right](t);\quad\text{if } t>0,\\
\\
\left[\left(\frac1m\left\{-\partial_t^2+(t-\xi)^2\right\}+a\alpha\right)u\right](t);\quad
\text{if }t<0.\end{array}\right.
\end{equation}
We denote by $\mu_1(a,m,\alpha;\xi)$ the first eigenvalue of
$H[a,m,\alpha;\xi]$ which is given by the min-max principle,
\begin{equation}\label{IIpVP}
\mu_1(a,m,\alpha;\xi)=\inf_{u\in B^1(\mathbb
R),u\not=0}\frac{q[a,m,\alpha;\xi](u)}{\|u\|^2_{L^2(\mathbb R)}}.
\end{equation}

We summarize in the next theorem the main results obtained in
\cite{Kach} concerning the above family of operators.

\begin{thm}\label{kach-AsymptAnal}
The following assertions hold.
\begin{enumerate}
\item Given $a>0$ and $m>0$, there exists a unique $\alpha_0(a,m)>0$
such that
$$\inf_{\xi\in\mathbb R}\mu_1(a,m,\alpha_0(a,m);\xi)=0\,.$$
Moreover, $\alpha_0(a,m)=1$ if $m\leq1$, and there exists a
universal constant $\Theta_0\in]0,1[$ such that
$\Theta_0<\alpha_0(a,m)<1$ if $m>1$, and
$$\displaystyle\lim_{m\to\infty}\alpha_0(a,m)=\Theta_0\,.$$
\item If $0<\alpha<\alpha_0(a,m)$, then $\displaystyle\inf_{\xi\in\mathbb
  R}
\mu_1(a,m,\alpha;\xi)>0$.
\item Given $a>0$, there exists $m_0>1$ such that for all
  $m\geq m_0$  the
function $\mathbb R\ni \xi\mapsto\mu_1(a,m,\alpha_0(a,m);\xi)$
admits a unique non-degenerate minimum.
\item For all $a>0$ and $m\geq m_0$,
there exists $\epsilon_0(m)>0$ such that
if
$$\alpha\in[\alpha_0(a,m)-\epsilon_0(m),\alpha_0(a,m)+\epsilon_0(m)]\,,$$ then the function $\mathbb
R\ni\xi\mapsto\mu_1(a,m,\alpha;\xi)$ admits a unique non-degenerate
minimum, denoted by $\xi(a,m,\alpha)$.
\item For all $m>0$ and $a>0$, $\displaystyle\inf_{\xi\in\mathbb
R}\mu_1(a,m,\alpha;\xi)\to0$ as $\alpha\to\alpha_0(a,m)$.

\end{enumerate}
\end{thm}

Let us mention that the universal constant $\Theta_0$ is the infimum
of the spectrum of the Neumann Schr\"odinger operator with unit
magnetic field in $\mathbb R\times\mathbb R_+$, and it holds that
$\frac12<\Theta_0<1$. We point also that the existence of the
constant $m_0>1$ in Theorem~\ref{kach-AsymptAnal} is nontrivial, and
is due to a fine asymptotic analysis of the eigenvalue (\ref{IIpVP})
as $m\to\infty$,
see \cite[Section~3.4]{Kach}.\\
For further use, we introduce the constant
\begin{equation}\label{beta}
\beta(a,m,\alpha)=\inf_{\xi\in\mathbb R}\mu_1(a,m,\alpha;\xi)\,,
\end{equation}
and the set
\begin{equation}\label{M}
\mathrm M(a,m,\alpha)=\{\xi\in\mathbb
R~:~\mu_1(a,m,\alpha;\xi)=\beta(a,m,\alpha)\}\,.
\end{equation}
We notice as a result of Theorem~\ref{kach-AsymptAnal} that
$\beta(a,m,\alpha)=0$ if and only if $\alpha=\alpha_0(a,m)$, and in
this
case $\mathrm M(a,m,\alpha)=\{\xi(a,m,\alpha)\}$ when $m\geq m_0$.\\
Let us also notice that it results from a simple application of the min-max
principle together with well known results concerning the harmonic
oscillator in $\mathbb R_+$,
\begin{equation}\label{beta*}
\beta(a,m,\alpha)+\alpha\geq\min\left(\Theta_0,\frac{\Theta_0}m
+(a+1)\alpha\right)\,,
\quad\forall~\alpha>0\,.
\end{equation}
Furthermore, let $f_{\alpha,\xi}^{a,m}$ be the normalized
eigenfunction associated with the eigenvalue (\ref{IIpVP}). We
introduce the constants
\begin{eqnarray}\label{II-C1}
C_1(a,m,\alpha;\xi)&=&\int_{\mathbb
R_+}(t-\xi)^3|f_{\alpha,\xi}^{a,m}(t)|^2dt+\frac1m\int_{\mathbb
R_-}(t-\xi)^3|f_{\alpha,\xi}^{a,m}(t)|^2dt\\
&&-\frac12\left(1-\frac1m\right)|f_{\alpha,\xi}^{a,m}(0)|^2\,,\\
b_1(a,m,\alpha;\xi)&=&\int_{\mathbb
R_+}|f_{\alpha,\xi}^{a,m}(t)|^2dt-a\int_{\mathbb
R_-}|f_{\alpha,\xi}^{a,m}(t)|^2dt\,.\label{II-b1}
\end{eqnarray}
If $m\geq m_0$, $\alpha$ fills the hypotheses of assertion (4) in
Theorem~\ref{kach-AsymptAnal} and $\xi=\xi(a,m,\alpha)$, then we
write simply:
\begin{eqnarray}\label{C1-b1}
&&C_1(a,m,\alpha)=C_1\big{(}a,m,\alpha;\xi(a,m,\alpha)\big{)}\,.\\
&&b_1(a,m,\alpha)=b_1\big{(}a,m,\alpha;\xi(a,m,\alpha)\big{)}\,,
\label{C1-b1*}
\end{eqnarray}
and if further, $\alpha=\alpha_0=\alpha_0(a,m)$, we introduce the
constant (appearing in the asymptotic formula \eqref{HC3-curv})
\begin{equation}\label{C1}
\mathcal C_1(a,m)=-\frac{C_1(a,m,\alpha_0)}{b_1(a,m,\alpha_0)}\,.
\end{equation}
Finally, we point out that 
it is proved in \cite[Section~3]{Kach},
$$\lim_{m\to\infty}C_1(a,m,\alpha_0)=-(1+6a\Theta_0^2)C_1^*\,,
\quad \lim_{m\to\infty}b_1(a,m,\alpha_0)=1\,,
$$
where $C_1^*>0$ is a universal constant (see \cite[Remark~1.4]{Kach}). Thus,
for large values of $m$, the constant in (\ref{C1-b1}) is negative.

Now, for $m>1$ and $\alpha>0$, we call $(\mathcal P_{m,\alpha})$ the
property below,
$$(\mathcal P_{m,\alpha})\quad\left\{
\begin{array}{l}
\mathrm M(a,m,\alpha)=\{\xi(a,m,\alpha)\}\textrm{ and }\xi(a,m,\alpha)>0\,,\\
\xi(a,m,\alpha) \textrm{ is a non-degenerate minimum point,}\\
C_1(a,m,\alpha)<0\quad{\rm and}\quad b_1(a,m,\alpha)>0\,.
\end{array}\right.$$
We introduce the constant $m_*\geq1$,
\begin{equation}\label{m*}
m_*=\inf\{m>1~:~\forall~m'\geq m\,,
~(\mathcal P_{m',\alpha}) \textrm{
  holds for }\alpha=\alpha(a,m')\}\,.
\end{equation}
In view of the result of Theorem~\ref{kach-AsymptAnal}, we may define
$m_*$ as above. We emphasize also that the constant $m_*$ depends on
$a$, but we omit that from the notation for the sake of simplicity. It
is conjectured in \cite{Kach} that $m_*=1$ for all $a>0$.\\
Now, given $a>0$ and $m>m_*$, we introduce
\begin{equation}\label{epsilon*}
\epsilon_*(m)=\sup\{\epsilon\in]0,{d_0}/2[~:~\forall~
\alpha\in[\alpha_0(a,m)-\epsilon,
\alpha_0(a,m)+\epsilon]\,,~(\mathcal P_{m,\alpha})\textrm{
holds}\}\,,
\end{equation}
where $d_0:=\Theta_0-\frac12>0$.\\
For the special case of a disc domain, one more constant will appear
to be relevant (this is $C_2(a,m,\alpha)$ introduced below). Given
$a>0$ and $m>m_*$, the lowest eigenvalue of the operator (\ref{II-H-})
for $\xi=\xi(a,m,\alpha)$ is $\beta(a,m,\alpha)$ introduced in
(\ref{beta}). Hence, the regularized resolvent, i.e. the operator
($\xi=\xi(a,m,\alpha)$),
\begin{equation}\label{reg-res}
R_0[a,m,\alpha](\phi)
=\left\{\begin{array}{l}
\left[H[a,m,\alpha;\xi)-\beta(a,m,\alpha)\right]^{-1}\phi\,,~{\rm
  if}~\phi\bot f_{\alpha,\xi}^{a,m}\,,\\
0\,,~{\rm otherwise}\,,\end{array}\right.
\end{equation}
is bounded in $L^2(\mathbb R)$. Letting $f=f_{\alpha,\xi}^{a,m}$ for
$\xi=\xi(a,m,\alpha)$, $\tilde w_m(t)=1$ if $t>0$ and $\tilde
w_m(t)=\frac1m$ if $t<0$, then  it is proved in \cite[Proposition~3.6]{Kach}
that the functions $f$ and $\tilde w_m f$ are orthogonal in
$L^2(\mathbb R)$, hence the integral
$$I_2(a,m,\alpha)=\int_0^\infty
(t-\xi)f\,
R_0(t-\xi)f\, d  t+\frac1{m^2}\int_{-\infty}^0
(t-\xi)f\,
R_0(t-\xi)f\, d  t$$
is positive. Notice that we write simply
$R_0$ for
the operator (\ref{reg-res}). Now, we introduce the constant,
\begin{equation}\label{C2}
C_2(a,m,\alpha)=\int_0^\infty|f(t)|^2\, d  t+
\frac1m
\int_{-\infty}^0|f(t)|^2\, d  t-4I_2(a,m,\alpha).
\end{equation}
Since $m>1$ and $\int_{\mathbb R}|f|^2\, d  t=1$, it is clear that
$C_2(a,m,\alpha)<1$. Recalling that
$$\alpha_0(a,m)>\Theta_0>\frac12\,,$$ we get for
$\alpha\in[\alpha_0(a,m)-\epsilon_*(m),\alpha_0(a,m)+\epsilon_*(m)]$
that
\begin{equation}\label{C2<2alpha}
\alpha_0(a,m)-\frac12C_2(a,m,\alpha)>\frac{d_0}2\,,\quad
d_0:=\Theta_0-\frac12\,.
\end{equation}

\subsection{Notation, hypotheses and announcement of Main Result}
We assume that $\Omega\subset\mathbb R^2$,
$\widetilde\Omega\subset\mathbb R^2$ are
open,
bounded and have  smooth boundaries such that
\begin{equation}\label{hyp-Omega}
\overline\Omega\subset\widetilde\Omega\,.
\end{equation}
Given $B>0$ and $\alpha>0$,
we denote by $P[B,\alpha]$ the
self-adjoint operator in $L^2(\widetilde\Omega)$  generated by the
quadratic form
\begin{equation}\label{qf-B}
H^1_{B\mathbf F}(\widetilde\Omega)\ni\phi\mapsto
Q[B,\alpha](\phi)=\int_{\widetilde\Omega}\left(
w_m(x)|(\nabla-iB\mathbf F)\phi|^2+\alpha BV_a(x)|\phi|^2\right)\, d  x\,,
\end{equation}
where $w_m$ and $V_a$ are introduced in (\ref{w-V}), and
\begin{equation}\label{H1-mag}
H^1_{B\mathbf F}(\widetilde\Omega)=\{u\in L^2(\widetilde\Omega)~:~
(\nabla-iB\mathbf F)u\in L^2(\mathbb R^2)\}\,.\end{equation}
Notice that since $\widetilde\Omega$ is bounded, $H^1_{B\mathbf
  F}(\widetilde\Omega)=H^1(\widetilde\Omega)$ and hence the form
domain of $Q$ is independent of $B$.

We introduce further
the lowest eigenvalue of the operator $P[B,\alpha]$:
\begin{equation}\label{mu1(B)}
\mu_1(B,\alpha)=\inf_{\substack{\phi\in H^1(\widetilde
 \Omega)\\
\phi\not=0}}\frac{Q[B,\alpha](\phi)}{\|\phi\|_{L^2(\widetilde\Omega)}^2}\,.
\end{equation}
By taking  $B=\kappa H$ and $\alpha=\kappa/H$, we see the connection with
the critical fields introduced in (\ref{HC3-loc}):
\begin{equation}\label{relation}
\mu^{(1)}(\kappa,H)=\mu_1(B,\alpha).\end{equation}
Furthermore, we put,
\begin{equation}\label{lambda-2}
\lambda_1(B,\alpha)=\mu_1(B,\alpha)+\alpha B\,.
\end{equation}

As an application of standard analytic perturbation theory, we have
the following proposition.

\begin{prop}\label{1O-PT}
Given $a>0$, $m>0$ and $\kappa>0$, the following one sided derivatives
\begin{eqnarray*}
&&\partial_H\,\mu^{(1)}_{\pm}(\kappa,H)=\lim_{\epsilon\to0_\pm}\frac{
\mu^{(1)}(\kappa,H+\epsilon)-\mu^{(1)}(\kappa,H)}\epsilon\,,\\
&&\partial_{B}\,\lambda_{1,\pm}(B,\alpha)=
\lim_{\epsilon\to0_\pm}\frac{
\lambda_1(B+\epsilon,\alpha)-\lambda_1(B,\alpha)}\epsilon\,,
\end{eqnarray*}
exist for all $H>0$, $B>0$ and $\alpha>0$\,, and
$$\partial_B\,\lambda_{1,+}(B,\alpha)\leq
\partial_B\,\lambda_{1,-}(B,\alpha)\,.$$
 Moreover, there
exist $L^2$-normalized ground states $\varphi_{H_+}$ and
$\varphi_{H_-}$ associated with $\mu^{(1)}(\kappa,H)$ such that,
\begin{equation}\label{cont-fun}
\partial_H\,\mu^{(1)}_{\pm}(\kappa,H)=\kappa\left(
\partial_{B}\,\lambda_{1,\pm}(B,\alpha)\bigg{|}
_{\substack{B=\kappa
    H\\
\alpha=\kappa/H}}
-\frac{\kappa}{H}\int_{\widetilde \Omega} (V_a(x)+1)|\varphi_{H_\pm}(x)|^2\, d  x
\right)\,.
\end{equation}
Here $V_a$ is introduced in (\ref{w-V}).
\end{prop}
\begin{proof}Let us explain briefly why the one sided derivatives
$\partial_B\lambda_{1,\pm}(B,\alpha)$ above exist; the one sided
derivatives in $H$
exist for exactly the same reason.

As we already mentioned, for real $B$, 
the operator $P[B,\alpha]$ is self adjoint,
has compact resolvent and its form domain, $H^1(\widetilde\Omega)$,
is independent from $B$. Then, one can show that there exists $z_0$
sufficiently small such that the operator
$$P[z,\alpha]=\nabla_{z\Fb}\cdot w_m\nabla_{z\Fb}+\alpha zV_a$$  is
analytic of type (B) in\footnote{Actually for $z_0$ small,
  $P[z,\alpha]$ is sectorial for all $z\in D_0$, its form domain is
  independent from $z$ and the expression of the quadratic form
  associated to 
$P[z,\alpha]$ is  analytic in $z$.}
  $D_0=\{z=x+iy~:~x>0,~|y|< z_0\}$, 
see \cite[p. 392]{Kato} for the definition of type (B) operators. In
particular, for a given $B>0$, the eigenvalue $\mu_1(B,\alpha)$ has
finite multiplicity $n$. Now, by analytic perturbation theory (see
\cite[Theorem~4.2, p. 395]{Kato}) applied to the family of operators
$D_0\ni z\mapsto P[z,\alpha],$ we get the existence of $\epsilon>0$ and
$2n$ analytic functions
$$(B-\epsilon,B+\epsilon)\ni z\mapsto\phi_j(z)\in
H^1(\widetilde\Omega)\setminus\{0\}, \quad(B-\epsilon,B+\epsilon)\ni
z\mapsto E_j(z)\in\mathbb R$$ such that
$$
P[z,\alpha]\left(\phi_j(z)\right)=E_j(z)\phi_j(z),\quad
E_j(B)=\mu_1(B,\alpha).$$ By choosing $\epsilon>0$ sufficiently
small, we get the existence of $j_\pm\in\{1,2,\cdots,n\}$ such that
\begin{align*}
&E_{j_+}(z)=\min_{j\in\{1,2,\cdots,n\}}E_j(z) \quad{\rm
for}~B<z<B+\epsilon,\\
&E_{j_-}(z)=\min_{j\in\{1,2,\cdots,n\}}E_j(z)\quad{\rm
for}~B-\epsilon<z<B\,.\end{align*} With this choice, it is clear
that $\partial_B\lambda_{1,\pm}(B,\alpha)=E_{j_\pm}'(B)+\alpha$.

Furthermore, we can choose the eigenfunctions $\phi_{j_\pm}(z)$ to be
$L^2$-normalized.
Now, the equality (\ref{cont-fun}) is obtained through differentiation of
the relation (in $z=0_\pm$)
$$\lambda_1(B+z,\alpha)=Q[B+z,\alpha](\phi_{j_\pm}(B+z))+\alpha (B+z)
\,,$$
and the application of the
chain rule.
\end{proof}

In this section, we shall work under the following hypothesis on the
constant $m$:
\begin{equation}\label{hypoth-m}
m> m_*\,,
\end{equation}
where $m_*>1$ is the constant introduced in
(\ref{m*}).\\
The next hypothesis is on the constant $\alpha$:
\begin{equation}\label{hypoth-alpha}
-\epsilon_*(m)\leq \alpha-\alpha_0(a,m)\leq
\epsilon_*(m)\,,\end{equation} where the constants
$\alpha_0(a,m)\in]\Theta_0,1[$ and $\epsilon_*(m)>0$ are introduced
in Theorem~\ref{kach-AsymptAnal} and (\ref{epsilon*}) respectively.

\begin{thm}\label{thm-monton}(General domains)\\
Under the  hypotheses \eqref{hypoth-m} and \eqref{hypoth-alpha}, if
$\Omega$ is not a disc,  the following holds
\begin{equation}\label{item1}
\lim_{B\to\infty}\left(\sup_{|\alpha-\alpha_0(a,m)|\leq\epsilon_*(m)}
\left|\partial_B\lambda_{1,\pm}(B,\alpha)-\beta(a,m,\alpha)-\alpha\right|\right)=0.
\end{equation}
\end{thm}

\begin{thm}\label{monoton-disc}(Disc domains)\\
Assume that $\Omega=D(0,1)$ is a disc. Given $a>0$ and $m> m_*$,
there exists a constant $B_0>1$ and a function $[B_0,\infty[\ni
B\mapsto g(B)$ satisfying $\displaystyle\lim_{B\to\infty}g(B)=0$,
such that if $B\geq B_0$ and $\alpha$ satisfies
(\ref{hypoth-alpha}), then
\begin{equation}\label{item2}
\partial_B\lambda_{1,+}(B,\alpha)\geq
    \alpha-\frac12C_2(a,m,\alpha)+g(B)\,.
\end{equation}
\end{thm}

In view of (\ref{cont-fun}), we get as corollary of
Theorems~\ref{thm-monton} and \ref{monoton-disc}:

\begin{thm}\label{corol-monton}
Let $a>0$ and $m>m_*$. Assume that
$\Omega\subset\mathbb R^2$ has a smooth and compact
boundary, and if $\Omega$ is a disc, assume in
addition that,
\begin{eqnarray}\label{condition-m}
&&\hskip-2cm\frac1a\left(a\alpha_0(a,m)+\frac12C_2(a,m,\alpha_0(a,m))\right)
\left(1-\frac{\Theta_0}{\alpha_0(a,m)}\right)\\
&&\hskip2cm<\alpha_0(a,m)-\frac12C_2(a,m,\alpha_0(a,m))\,.\nonumber
\end{eqnarray}
Then  there exists $\kappa_0>0$ such that, for all $\kappa\geq
\kappa_0$,
there is a unique $H_*(\kappa)>0$ solving the equation
$$\mu^{(1)}(\kappa,H_*(\kappa))=0\,,$$
where $\mu^{(1)}(\kappa,H)$ is introduced in (\ref{EV}).\\
Moreover,
for all $H>H_*(\kappa)$ and $\kappa\geq\kappa_0$,
$\mu^{(1)}(\kappa,H)>0$.
\end{thm}
\begin{rem}\label{rem:Conclusion}
Thanks to the asymptotic result of Theorem~\ref{kach-AsymptAnal},
the condition
  (\ref{condition-m}) is fulfilled for large values of $m$. Thus
  Theorem~\ref{kach-AsymptAnal} and Theorem~\ref{corol-monton} imply Theorem~\ref{thm:identification}.
\end{rem}

\begin{proof}[Proof of Theorem~\ref{corol-monton}] From
\cite{Kach}, we get a constant $\kappa_0$ and a positive function
$\delta(\kappa)$ such that $\lim_{\kappa\to\infty}\delta(\kappa)=0$
and
$$
\forall~\kappa\in[\kappa_0,\infty[\,,\quad\mu^{(1)}(\kappa,H)=0\implies
\left|H-\frac{\kappa}{\alpha_0(a,m)}\right|\leq
\kappa\delta(\kappa)\,.$$ It is moreover proved that one may find a
solution $H>0$ such that $\mu^{(1)}(\kappa,H)=0$.

 Let
$H_*(\kappa)=\min\{H>0~:~\mu^{(1)}(\kappa,H)=0\}$. It is sufficient
to show that $\mu^{(1)}(\kappa,H)>0$ for any magnetic field $H$
satisfying \begin{equation}\label{f(B)} H_*(\kappa)<H\leq
\frac{\kappa}{\alpha_0(a,m)}+\kappa\delta(\kappa)\,.\end{equation}
Thus, for such a magnetic field, we put $B=B(\kappa;H)=\kappa H$ and
$\alpha=\alpha(\kappa;H)=\kappa/H$
so that we can pick $\kappa_0>0$ sufficiently large such that the
hypothesis (\ref{hypoth-alpha}) is valid, and hence the results of
Theorems~\ref{thm-monton} and \ref{monoton-disc} hold.

Let us show that the function $H\mapsto \mu_1(\kappa,H)$ is strictly
increasing in the interval
\begin{equation}\label{interval}
\left]\frac{\kappa}{\alpha_0(a,m)}-\kappa\delta(\kappa)\,,\,
\frac{\kappa}{\alpha_0(a,m)}+\kappa\delta(\kappa)\right[\,.\end{equation}
We assume that $\kappa\geq\kappa_0$. From (\ref{cont-fun}), we write
\begin{equation}\label{correction1}
\partial_H\,\mu^{(1)}_{\pm}(\kappa,H)
=\kappa\left(\partial_B\lambda_{1,\pm}(B,\alpha)
-\alpha\int_{\widetilde \Omega}(V_a(x)+1)|\phi_{B,\alpha}|^2\, d  x\right)\,,
\end{equation}
where $\phi_{B,\alpha}$ is an $L^2$-normalized eigenfunction
associated with the lowest
eigenvalue $\mu_1(B,\alpha)$.\\
Assume first that $\Omega$ is not a disc. Invoking
Theorem~\ref{thm-monton}, we get, provided that $\kappa_0$ is large
enough,
\begin{equation}\label{correction2}
\partial_H\,\mu^{(1)}_{\pm}(\kappa,H)\geq
\kappa\alpha\left(1- \int_{\widetilde\Omega}(V_a(x)+1)|\phi_{B,\alpha}|^2\, d 
x+o(1)\right)\,,
\end{equation}
where we use also that for $H$ in the interval (\ref{interval}),
$\beta(a,m,\alpha)=o(1)$ as $\kappa\to\infty$ (see
Theorem~\ref{kach-AsymptAnal}).

Since $\phi_{B,\alpha}$ is
normalized in $L^2(\widetilde\Omega)$ and an eigenfunction
associated with $\mu_1(B,\alpha)$, we write,
\begin{align*}
\alpha\Big(1-\int_{\widetilde\Omega}(V_a(x)&+1)|\phi_{B,\alpha}|^2\, d 
x\Big)\\
&=\alpha\int_\Omega|\phi_{B,\alpha}|^2\, d 
x-a\alpha\int_{\widetilde\Omega\setminus\Omega}|\phi_{B,\alpha}|^2\, d  x\\
&=B^{-1}\left(\int_{\widetilde\Omega} w_m(x)|(\nabla-iB\mathbf
F)\phi_{B,\alpha}|^2\, d  x-\mu_1(B,\alpha)\right)\,.
\end{align*}
Using the min-max variational principle, we infer from the above,
$$ \alpha\left(1-\int_{\widetilde\Omega}(V_a(x)+1)|\phi_{B,\alpha}|^2\, d 
x\right)\geq B^{-1}\left(\frac1m\mu^N(B)+\mu_1(B,\alpha)\right)\,,$$
where $\mu^N(B)$ is the lowest eigenvalue of the Neumann magnetic
Laplacian\break $-(\nabla-iB\mathbf F)^2$ in $\widetilde\Omega$.\\
It is a standard result that $\mu^N(B)=\Theta_0B+o(B)$ as
$B\to\infty$ (see \cite{HeMo1}), and under the condition that $H$
remains in the interval (\ref{interval}), it is proved in
\cite{Kach} that $\mu_1(B,\alpha)=o(B)$ as $B\to\infty$ (see
Proposition~\ref{kach-thm1} below). Thus, provided that $H$ is in
the interval (\ref{interval}) and $\kappa_0$ is large enough, we get
$$\alpha\left(1-\int_{\widetilde\Omega}(V_a(x)+1)|\phi_{B,\alpha}|^2\, d 
x\right)\geq \frac{\Theta_0}m+o(1),$$ and upon substitution in
(\ref{correction2}), we get when $\Omega$ is not a disc
$$\partial_H\mu^{(1)}_{\pm}(\kappa,H)>0\,,\quad\forall~\kappa\geq\kappa_0\,.$$
Now, we assume that $\Omega$ is a disc. In this case,
Theorem~\ref{monoton-disc} and (\ref{correction1}) together yield,
\begin{equation}\label{correction3}
\partial_H\,\mu^{(1)}_{\pm}(\kappa,H)\geq
\kappa\left(\alpha-\frac12C_2(a,m,\alpha) -\alpha
(a+1)\int_{\widetilde\Omega\setminus\Omega}|\phi_{B,\alpha}|^2\, d 
x+o(1)\right).
\end{equation}
Again, using $\int_{\widetilde\Omega}|\phi_{B,\alpha}|^2 d  x=1$, we
write,
\begin{eqnarray*}
&&\hskip-0.5cm\alpha-\frac12C_2(a,m,\alpha) -\alpha
(a+1)\int_{\widetilde\Omega\setminus\Omega}|\phi_{B,\alpha}|^2\, d 
x\\
&&=\left(\alpha-\frac12C_2(a,m,\alpha)\right)\int_\Omega|\phi_{B,\alpha}|^2\, d 
x-\left(a\alpha+\frac12C_2(a,m,\alpha)\right)
\int_{\widetilde\Omega\setminus\Omega}|\phi_{B,\alpha}|^2\, d  x\,.
\end{eqnarray*}
Now, we use that $\phi_{B,\alpha}$ is an eigenfunction associated
with $\mu_1(B,\alpha)$. From the identity
$Q[B,\alpha](\phi_{B,\alpha})=\mu_1(B,\alpha)$ and the min-max
principle, we deduce,
$$a\int_{\widetilde\Omega\setminus\Omega}
|\phi_{B,\alpha}|^2\, d  x\leq \left(1-
\frac{\mu^N(B,\Omega)}{\alpha}\right)\int_\Omega|\phi_{B,\alpha}|^2\, d 
x\,.$$ Here $\mu^N(B,\Omega)$ is the first eigenvalue of the Neumann
magnetic Laplacian $-(\nabla-iB\Ab)^2$ in $\Omega$. Using again the
asymptotic behavior of $\mu^N(B)$ as $B\to\infty$ and the condition
that $H$ remains in (\ref{interval}), we deduce,
$$a\int_{\widetilde\Omega\setminus\Omega}
|\phi_{B,\alpha}|^2\, d  x\leq \left(1-
\frac{\Theta_0}{\alpha_0(a,m)}+o(1)\right)\int_\Omega|\phi_{B,\alpha}|^2\, d 
x\,.$$ Let us define,
\begin{eqnarray*}
\delta_0&=&\alpha-\frac12C_2(a,m,\alpha_0(a,m))\\
&&-\frac1a\left(a\alpha+\frac12
C_2(a,m,\alpha_0(a,m))\right)\left(1-\frac{\Theta_0}{\alpha_0(a,m)}\right)\,.
\end{eqnarray*}
Under the condition (\ref{condition-m}), $\delta_0>0$. Under the
hypothesis that $H$ remains in the interval (\ref{interval}), we get
provided that $\kappa_0$ is large enough,
$$
\alpha-\frac12C_2(a,m,\alpha) -\alpha
(a+1)\int_{\widetilde\Omega\setminus\Omega}|\phi_{B,\alpha}|^2\, d 
x\geq \frac{\delta_0}2\,,$$ for all $\kappa\geq\kappa_0$.
Substituting in (\ref{correction3}), we get the desired result that
the function $H\mapsto \mu^{(1)}(\kappa,H)$ is strictly increasing
in the interval (\ref{interval}) for all $\kappa\geq\kappa_0$.
\end{proof}

To prove  Theorem~\ref{thm-monton} we deal separately with the case
where the domain $\Omega$ is a disc and the case where it is not.
This will be the subject of the next sections, but we review first
some of the known results concerning $\mu_1(B,\alpha)$.

\subsection{Some facts concerning $\mu_1(B,\alpha)$}
Let us recall the main results obtained in \cite{Kach} for
$\mu_1(B,\alpha)$, the bottom of the spectrum of the operator
$P[B,\alpha]$ associated with the quadratic form (\ref{qf-B}). We
mention that in \cite{Kach} we deal with $\mu^{(1)}(\kappa,H)$ rather
than $\mu_1(B,\alpha)$, but due to (\ref{relation}) we get equivalent
results for $\mu_1(B,\alpha)$.

We start by giving the  leading order term of $\mu_1(B,\alpha)$ as
$B\to\infty$ (\cite[Section~5]{Kach}).

\begin{prop}\label{kach-thm1}
Given $a>0$ and $m>0$, there exist a constant $B_0>0$ and a
function $[B_0,\infty[\ni B\mapsto g(B)\in\mathbb R_+$ satisfying
$\displaystyle\lim_{B\to\infty}g(B)=0$ such that, for any
$\alpha>0$ fulfilling (\ref{hypoth-alpha}),
the following asymptotic holds:
$$\left|\mu_1(B,\alpha)
-\beta(a,m,\alpha) B\right|\leq g(B)B\,,\quad \forall~B\geq B_0\,,$$
where $\beta(a,m,\alpha)$ is introduced in (\ref{beta}).
\end{prop}

Letting $\phi_{B,\alpha}$ be a ground state of $P[B,\alpha]$, it is
proved in \cite[Section~6]{Kach} that for $m>1$, $\phi_{B,\alpha}$
is localized near the
interior boundary $\partial\Omega$ in the sense precised in the next lemma.

\begin{lem}\label{kach-decay}
Under the hypotheses of Proposition~\ref{kach-thm1}, given $a>0$,
$m>1$
and an integer $k\in\mathbb N$, there exist
positive constants $B_k$ and $C_k$ such that,
\begin{eqnarray}
&&\int_{\widetilde\Omega} [{\rm
      dist}(x,\partial\Omega)]^k\,|\phi_{B,\alpha}(x)|^2\, d  x\leq
      \frac{C_k}{\sqrt{B}^{\,k}}
\left\|\phi_{B,\alpha}\right\|_{L^2(\widetilde\Omega)}^2\,,\label{decay-1}\\
&&\int_{\widetilde\Omega} [{\rm
      dist}(x,\partial\Omega)]^k\,|(\nabla-iB\mathbf F)\phi_{B,\alpha}(x)|^2\, d  x\leq
      \frac{C_k}{\sqrt{B}^{\,k-2}}
\left\|\phi_{B,\alpha}\right\|_{L^2(\widetilde\Omega)}^2\,,\label{decay-2}
\end{eqnarray}
for all $B \geq B_k$.
\end{lem}

Under the additional hypotheses (\ref{hypoth-m}) and
(\ref{hypoth-alpha}), i.e. $\alpha$ close to the constant
$\alpha_0(a,m)$ and $m$ strictly larger than the constant $m_*\geq
1$ introduced in \eqref{m*}, we get a two term asymptotic
expansion for the eigenvalue $\mu_1(B,\alpha)$ (see
\cite[Propositions~7.1~\and~7.3]{Kach}).

\begin{thm}\label{kach-thm2}
Assume that $a>0$ and $m$ satisfies (\ref{hypoth-m}). There exist
a constant $B_0>0$ and
a function $[B_0,\infty[\mapsto g(B)\in\mathbb R_+$ satisfying
    $\displaystyle\lim_{B\to\infty} g(B)=0$ such that, if $\alpha$
fulfills (\ref{hypoth-alpha}), then the following asymptotic holds:
$$\left|\mu_1(B,\alpha)-\left(\beta(a,m,\alpha)\,B
+C_1(a,m,\alpha)\,(\kappa_{\rm
  r})_{\rm max}\sqrt{B}\right)\right|\leq g(B)\sqrt{B}\,,
\quad\forall~B\geq B_0\,,$$
where $\beta(a,m,\alpha)$ is introduced in (\ref{beta}),
  $C_1(a,m,\alpha)$
the negative constant introduced in (\ref{C1-b1}),
  $\kappa_{\rm r}$ denotes the scalar curvature of $\partial\Omega$
  and $(\kappa_{\rm r})_{\rm max}$ denotes the maximum of $\kappa_{\rm r}$.
\end{thm}
It should be noticed that there exist negative constants $c_1$ and $c_2$ such
that, for $\alpha$ satisfying (\ref{hypoth-alpha}),
$$c_1\leq C_1(a,m,\alpha)\leq c_2<0\,.$$

\subsection{Proof of Theorem~\ref{thm-monton}: General domains}
We assume in this section that the domain $\Omega$ is not a disc,
still with smooth and compact boundary. Thus, the set
\begin{equation}\label{Pi}
  \Pi=\{x\in \partial\Omega~:~\kappa_{\rm r}(x)= (\kappa_{\rm r})_{\rm
    max}\}\end{equation}
is not identical to $\partial\Omega$, i.e. $\Pi\not=\partial\Omega$.\\
In this case, we can deduce from the asymptotic expansion in
Theorem~\ref{kach-thm2} that any  ground state $\phi_{B,\alpha}$ of
$P[B,\alpha]$ is localized near the set $\Pi$. One weak version of
expressing this rough statement is through the estimate of the
following lemma, see \cite{HeMo1,HePa} for sharper statements.

\begin{lem}\label{HeMo}
  Under the hypotheses of Theorem~\ref{kach-thm2}, for all
  $\epsilon_0>0$ and $N\in\mathbb N$, there exist  constants $B_0>0$
  and $C_N$ such that,
$$\int_{\{{\rm dist}(x,\Pi)\geq\epsilon_0\}}|\phi_{B,\alpha}(x)|^2\, d 
x\leq C_NB^{-N}\,,\quad\forall~B\geq B_0\,.$$
\end{lem}
\begin{proof}[Sketch of Proof of Lemma~\ref{HeMo}.]
  We provide some details concerning the derivation of the estimate of
  Lemma~\ref{HeMo}.
  Using the decay estimates of Lemma~\ref{kach-decay}, one can get the
  lower bound\,\footnote{ This is a consequence of combining
    Proposition~5.3.3 and Theorem~4.3.8 in \cite{KachTh}. More details
    are given also in \cite[pp. 179-180]{KachTh}},
  \begin{equation}\label{pr:dec-cu}
    Q[B,\alpha](\phi_{B,\alpha})\geq
    \int_{\widetilde\Omega}U_{B}(x)|\phi_{B,\alpha}(x)|^2\, d 
    x\,,\end{equation}
  where
$$
U_B(x)=\left\{
  \begin{array}{ll}
    (\gamma-\alpha) B\,,&{\rm if}~{\rm dist}(x,\partial\Omega)\geq
    2B^{-1/6}\,,\\
    \beta B+C_1\kappa_{\rm r}(p(x))B^{1/2}-C_0
    B^{1/3}\,,& {\rm if}~{\rm dist}(x,\partial\Omega)\leq
    2B^{-1/6}\,.
  \end{array}\right.
$$
Here $\gamma\in]0,1[$ is a given constant, $\beta=\beta(a,m,\alpha)$,
$C_1=C_1(a,m,\alpha)$, the constant $C_0>0$
depends on $a>0$, $m>m_*$ and $\gamma$, and $p(x)\in\partial\Omega$ is
defined
by $|p(x)-x|={\rm dist}(x,\partial\Omega)$.\\
On the other hand, it follows from Theorem~\ref{kach-thm2} that
\begin{multline}\label{pr:dec-cu*}
  Q[B,\alpha](\phi_{B,\alpha})\leq \\
  \int_\Omega\left(\beta(a,m,\alpha) B+C_1(a,m,\alpha)(\kappa_{\rm
      r})_{\rm max}B^{1/2}+g(B)B^{1/2}\right)|\phi_{B,\alpha}|^2\, d 
  x\,.
\end{multline}
Combining (\ref{pr:dec-cu}) and (\ref{pr:dec-cu*}), one can prove the
following exponential decay using standard (Agmon-type) techniques,
(see \cite[Theorem~6.5.4]{KachTh} for details)
$$
\left\| \exp\left(\delta_0B^{1/4}{\rm
      dist}(x,\Pi)\right)\phi_{B,\alpha}\right\|
_{L^2(\widetilde\Omega)}\leq
M\|\phi_{B,\alpha}\|_{L^2(\widetilde\Omega)}\,,
$$
where $\delta_0>0$ and $M$ are constants depending on $a$ and $m$. Now
the estimate of the lemma is just a consequence of the above
exponential decay of $\phi_{B,\alpha}$.
\end{proof}

We start now by following ideas of Fournais-Helffer in \cite{FoHe06}
to prove monotonicity of the eigenvalue
\begin{equation}\label{lambda-1}
  \lambda_1(B,\alpha)=\mu_1(B,\alpha)+\alpha B\,.
\end{equation}
Notice that $\lambda_1(B,\alpha)$ is the eigenvalue of the operator
$$\widetilde P[B,\alpha]=P[B,\alpha]+\alpha  B$$
associated with the quadratic form
\begin{equation}\label{qf-B1}
  \phi\mapsto\widetilde Q[B,\alpha](\phi)=Q[B,\alpha](\phi)
  +\alpha B\|\phi\|^2_{L^2(\widetilde \Omega)}\,,
\end{equation}
where $Q[B,\alpha]$ is the quadratic form introduced in (\ref{qf-B}),
and both operators $\widetilde P[B,\alpha]$ and $P[B,\alpha]$ admit
the same ground states. With this point of view, it is more convenient
to adapt the proof of \cite{FoHe06}.

We proceed to prove Theorem~\ref{thm-monton}.
\subsubsection{Gauge transformation}
Using an idea from \cite{FoHe06}, we will apply a gauge transformation
and work with a new magnetic potential $\Ab$ (instead of $\mathbf
F$). In the new gauge, we will have that $|\Ab\phi_{B,\alpha}|$ is
small in the $L^2$-sense as $B\rightarrow \infty$. Notice that by
Lemma~\ref{HeMo} $\phi_{B,\alpha}$ is localized near the set $\Pi$, so
it suffices to find a gauge where $\Ab=0$ on $\Pi$.

To this end, we introduce adapted coordinates near the boundary of
$\Omega$. For a sufficiently small $t_0>0$, we introduce the open set
$$\Omega(t_0)=\{x\in\mathbb R^2~:~{\rm
  dist}(x,\partial\Omega)<t_0\}.$$ Let $s\mapsto\gamma(s)$ be the
arc-length parametrization of $\partial\Omega$ and $\nu(s)$ the unit
inward normal of $\partial\Omega$ at
$\gamma(s)$.\\
When $t_0$ is sufficiently small, the transformation
\begin{align}
  \label{eq:19}
  \Phi:
  \left[-\frac{|\partial\Omega|}2,\frac{|\partial\Omega|}2\right[\,\times
  ]-t_0,t_0[\ni(s,t)\mapsto \gamma(s)+t\nu(s)\in \Omega(t_0)
\end{align}
becomes a difeomorphism whose Jacobian is $|D\Phi|=1-t\kappa_{\rm
  r}(s)$. For $x\in\Omega(t_0)$, we put
$$\Phi^{-1}(x)=(s(x),t(x))$$
and we get in particular that
$$t(x)={\rm dist}(x,\partial\Omega)\quad{\rm in~}\Omega,\quad
t(x)=-{\rm dist}(x,\partial\Omega)\quad{\rm outside~}\Omega\,.$$

Just as in \cite[Lemma~2.5]{FoHe06}, we get the following lemma.
\begin{lem}\label{FoHe-lem-gauge}
  Let $0<\epsilon\leq \frac12\min(t_0,|\partial\Omega|)$,
  $x_0=\Phi(s_0,0)\in\partial\Omega$ and
$$\Omega(\epsilon,s_0)=\{x=\Phi(s,t)\in\widetilde\Omega~:~|t|
\leq\epsilon,~|s-s_0|\geq \epsilon\}\,.$$ Then there exists a function
$\varphi\in C^\infty_0(\mathbb R^2)$ such that $\Ab=\mathbf
F+\nabla\varphi$ satisfies
$$|\Ab(x)|\leq C\,{\rm dist}(x,\partial\Omega)\quad{\rm
  in}~\Omega(\epsilon,s_0)\,,$$ where $C>0$ depends only on $\Omega$.
\end{lem}

\begin{proof}[Proof of Theorem~\ref{thm-monton}]~\\
We have now all the prerequisites needed to apply the argument of
Fournais-Helffer
\cite{FoHe06}. We include the details for the reader's convenience. 

Recall that, for $\epsilon>0$ sufficiently small and
$z\in[B,B+\epsilon[$, we associate to $\lambda_1(z,\alpha)$ an
analytic branch of eigenfunctions $z\mapsto\phi_{z,\alpha}$ such that
$$\widetilde P[z,\alpha]\phi_{z,\alpha}=\lambda_1(z,\alpha)
\phi_{z,\alpha}\quad\forall~z\in[B,B+\epsilon[\,.$$ We may also in
addition assume that
$\|\phi_{z,\alpha}\|^2_{L^2(\widetilde\Omega)}=1$ for all
$z\in[B,B+\epsilon[$.

Since $\Omega$ is not a disc, the set $\Pi$ in (\ref{Pi}) is different
from $\Omega$, and we can find
$\epsilon_0\in\frac12\min(t_0,|\partial\Omega|)$ and
$s_0\in\partial\Omega$ such that
\begin{equation}\label{conseq-genD}
[s_0-2\epsilon_0,s_0+2\epsilon_0]\cap \Pi=\emptyset\,.
\end{equation}
Let $\Ab$ and $\varphi$ be respectively  the vector field and the
real-valued function defined in Lemma~\ref{FoHe-lem-gauge}. Let
$\widehat P[z,\alpha]$ be the self adjoint operator associated to
the quadratic form
$$H^1(\widetilde \Omega)\ni u\mapsto \int_{\widetilde \Omega}\left(
w_m(x)|(\nabla-iz\Ab)u|^2+\alpha z(V_a(x)+1)|u|^2\right) d  x\,$$
where $w_m$ and $V_a$ are introduced in (\ref{w-V}). Notice that the
operators $\widehat P[z,\alpha]$ and $\widetilde P[z,\alpha]$ are
unitary equivalent~: $\widehat P[z,\alpha]u=e^{iz\varphi}\widetilde
P[z,\alpha]e^{-iz\varphi}u$. Hence,
$$[B,B+\epsilon[\,\ni z\mapsto\widehat\phi_{z,\alpha}:=
e^{iz\varphi}\phi_{z,\alpha}$$ is an analytic branch of
eigenfunctions associated to the eigenvalue $\lambda_1(z,\alpha)$ of
the operator $\widehat P[z,\alpha]$.

Calculating,
\begin{align*}
&\partial_B\lambda_{1,+}(B)=\left.\frac{d}{dz}\left(\widehat
Q[B+z,\alpha](\widehat
\phi_{z,\alpha})\right)\right|_{z=0_+}\\
&=\int_{\widetilde\Omega} 2\,w_m(x)\,{\rm
Re}\left\langle-i\Ab\widehat\phi_{B,\alpha},(\nabla-iB\Ab)\widehat\phi_{B,\alpha}\right\rangle
+\alpha\left(V_a(x)+\alpha\right)|\widehat\phi_{B,\alpha}(x)|^2\, d 
x\\
&\qquad\qquad+2\,{\rm Re}\,\,\widehat
Q[B,\alpha]\Big(\frac{ d \,\widehat\phi_{B+z,\alpha}}{ d 
z}\Big|_{z=0_+}\,,\,\widehat\phi_{B,\alpha}\Big),
\end{align*}
where the last term on the r.h.s. above vanishes, since
$\widehat\phi_{z,\alpha}$ is normalized with respect to the $L^2$
norm. Thus, for an arbitrary $\zeta\in\mathbb R\setminus\{0\}$, we
may express $\partial_B\lambda_{1,+}(B)$ in the following way
$$
\partial_B\lambda_{1,+}(B)=\frac{\widehat
Q[B+\zeta,\alpha](\widehat\phi_{B,\alpha})-\widehat
Q[B,\alpha](\widehat\phi_{B,\alpha})}\zeta-\zeta\int_{\widetilde\Omega}|\Ab
\widehat\phi_{B,\alpha}|^2\, d  x,
$$
and invoking the min-max principle we get further when $\zeta>0$
\begin{equation}\label{final-step}
\partial_B\lambda_{1,+}(B)\geq
\frac{\lambda_1(B+\zeta,\alpha)-\lambda_1(B,\alpha)}\zeta-\zeta\int_{\widetilde\Omega}|\Ab\phi_{B,\alpha}|^2\, d 
x.\end{equation} By Lemma~\ref{FoHe-lem-gauge}, we may write,
$$\int_{\widetilde\Omega}|\Ab\phi_{B,\alpha}|^2\, d  x\leq
C\int_{\widetilde\Omega}|{\rm
dist}(x,\partial\Omega)|^2|\phi_{B,\alpha}|^2\, d 
x+\|\Ab\|_{L^\infty(\widetilde\Omega)}^2\int_{\widetilde\Omega\setminus\Omega(s_0,\epsilon_0)}|\phi_{B,\alpha}|^2\, d 
x\,.$$ By our choice of $s_0$ and $\epsilon_0$,
$\widetilde\Omega\setminus\Omega(s_0,\epsilon_0)$ is away from
boundary points  of maximal curvature. Thus,
invoking Lemmas~\ref{decay-1} and \ref{HeMo}, we obtain a constant
$B_0>0$ depending only on $a$ and $m$ such that
$$\int_{\widetilde\Omega}|\Ab\phi_{B,\alpha}|^2\, d  x\leq
CB^{-1}\quad{\rm for~}B\geq B_0\,,$$ and  we implement this last
estimate in (\ref{final-step}). Now, choosing $\zeta=\eta B$ in
(\ref{final-step}) with $\eta\in]0,1[$ being arbitrary, we get from
Proposition~\ref{kach-thm1},
\begin{eqnarray}\label{der+}
\partial_B\lambda_{1,+}(B)&\geq&\beta(a,m,\alpha)+\alpha\\
 &&-\frac{1+\eta}{\eta}g((1+\eta)B)-\frac1\eta g(B)-C\eta\,,
\nonumber\end{eqnarray} where $g$ is independent from $\alpha$, and
$g(B)\to0$ as $B\to\infty$.

Applying the same argument to the left derivative
 $\partial_B\lambda_{1,-}(B)$, we get (the inequality being reversed
 since $z<0$ in this case),
\begin{eqnarray}\label{der-}
\partial_B\lambda_{1,-}(B)&\leq&\beta(a,m,\alpha)+\alpha\\
 &&+\frac{1+\eta}{\eta}g((1+\eta)B)+\frac1\eta g(B)+C\eta\,.
\nonumber\end{eqnarray}
Recall that analytic perturbation theory gives
 $\partial_B\lambda_{1,+}(B)\leq\partial_B\lambda_{1,-}(B)$ for all
 $B$. Therefore, (\ref{der+}) and (\ref{der-}) when combined together yield,
\begin{align*}
\limsup_{B\to\infty}
\left(\sup_{|\alpha-\alpha_0(a,m)|\leq\epsilon_*(m)}
\left|\partial_B\lambda_{1,\pm}(B,\alpha)-\beta(a,m,\alpha)-\alpha\right|\right)
\leq C\eta\,.
\end{align*} 
Taking $\eta\to0_+$ above, we get
\begin{align*}\lim_{B\to\infty}
\left(\sup_{|\alpha-\alpha_0(a,m)|\leq\epsilon_*(m)}
\left|\partial_B\lambda_{1,\pm}(B,\alpha)-\beta(a,m,\alpha)-\alpha\right|\right)=0\,.   
\end{align*}
\end{proof}

\subsection{Proof of Theorem~\ref{monoton-disc}: Disc domains}
In order to handle disc domains, we need a refined asymptotic
expansion of the first eigenvalue $\mu_1(B,\alpha)$ as $B\to\infty$.
Let us introduce some notation. Given $a>0$ and $m>m_*$,
we introduce,
\begin{equation}\label{delta(a,m)}
\delta(n,B)=n-\frac{1}2B-\xi(a,m,\alpha)\sqrt{B}\,.
\end{equation}

\begin{thm}\label{fe-disc}
Assume that $\Omega=D(0,1)$. Given $a>0$ and $m> m_*$, there exist a
constant $B_0>1$ and a function $[B_0,\infty[\ni B\mapsto g(B)$
satisfying $\displaystyle\lim_{B\to\infty}g(B)=0$, and if $\alpha$
satisfies (\ref{hypoth-alpha}),  there exist real constants
    $$\delta_0=\delta_0(a,m,\alpha)\,,\quad \mathcal C_0=
\mathcal C_0(a,m,\alpha)\,,$$
such that, for all $B\geq B_0$, the following expansion holds,
\begin{align}
  \label{eq:18}
  \left|\mu_1(B,\alpha)-\left(\beta(a,m,\alpha)B-C_1(a,m,\alpha)
\sqrt{B}
+C_2(a,m,\alpha)\left(\Delta_B^2+\mathcal C_0\right)
\right)\right| \nonumber \\
\leq g(B)\,.
\end{align}
Here $C_2(a,m,\alpha)$ is introduced in (\ref{C2}) and
$$\Delta_B=\inf_{n\in\mathbb Z}|\delta(n,B)-\delta_0(a,m,\alpha)|\,.$$
\end{thm}

The proof of Theorem~\ref{fe-disc} is very close to that
of Theorem~2.5 in \cite{FoHe} and
relies strongly on the fact that the eigenvalue (\ref{mu1(B)}) admits
a non-degenerate minimum in $\xi$. For the convenience of the reader,
we show in the appendix how the proof of \cite{FoHe} gives
Theorem~\ref{fe-disc}.

Recall the eigenvalue $\lambda_1(B,\alpha)$ introduced in
\eqref{lambda-1}. In view of the result of Theorem~\ref{fe-disc}
above, Theorem~\ref{monoton-disc} follows directly as
\cite[Theorem~2.5]{FoHe}. We omit thus the details.

\section{Exponential decay of order parameters}
\label{decay}

The objective of this section is to prove that, for $m>1$, energy
 minimizing order parameters decay exponentially away from the
 boundary. This estimate will be useful for forthcoming works
 dedicated to finer properties of energy minimizers.

The main theorem is the following.

\begin{thm}\label{thm:Agmon}
Assume that $a>0$, $m>1$ and let $b>0$ be a given constant.
There exist positive constants $M$, $C$, $\varepsilon$ and $\kappa_0$
such that, if $(\psi,\Ab)$ is a  solution of \eqref{GL-eq}
 and the
magnetic field verifies
$$\frac{H}\kappa\geq 1+b\,,$$
then
\begin{equation}\label{exponentialdecay}
\int_{\widetilde\Omega}
e^{2\varepsilon\sqrt{\kappa H}\,t(x)}\left(
|\psi|^2+\frac1{\kappa H}\left|(\nabla-i\kappa H\Ab)\psi\right|^2\right)
\, d  x\leq C\int_{\{t(x)\leq\frac{M}{\kappa H}\}}|\psi|^2\, d  x\,.
\end{equation}
Here $t(x):={\rm dist}(x,\partial\Omega)$.
\end{thm}

When $m> m_*$ and the magnetic field $H=H_{C_3}(\kappa)-o(1)$,
one should be able to prove a finer localization of $\psi$, near
boundary points with maximal curvature. Actually, an estimate similar
to Lemma~\ref{HeMo} for the linear problem should also be valid for
the solution $\psi$ of the non-linear Ginzburg-Landau problem.
However,
to establish such an estimate will require a rather very technical
work following previous papers \cite{HePa, LuPa}, so that we do not
carry it out. Let us only notice here that---just as for the linear
case---such a localisation estimate is essentially due to the asymptotic
expansion of the first eigenvalue stated in Theorem~\ref{kach-thm2}.

\begin{proof}[Proof of Theorem~\ref{thm:Agmon}]~\\
If $\psi\equiv0$, then the estimate of the theorem is evidently
true. Thus, thanks to Theorem~\ref{thm-GP-type}, we may restrict
ourselves to magnetic fields $H$ satisfying
$$H\leq C\kappa$$
for some sufficiently large positive constant $C$.
 
Let $\widetilde \chi\in C^{\infty}(\R)$ be a non-decreasing function with
\begin{align}
  \label{eq:6}
  \widetilde \chi = 1, \quad \text{ on } \quad [1,\infty), \qquad
  \qquad \widetilde \chi = 0, \quad
  \text{ on } (-\infty, 1/2].
\end{align}
Define $\chi$ on $\widetilde \Omega$ by $\chi(x) =
\widetilde \chi(
\sqrt{\kappa H} t(x))$.
Define furthermore, the weighted localisation function $f$ by
\begin{align}
  \label{eq:8}
  f(x) =  \chi(x)  \exp(\varepsilon \sqrt{\kappa H} t(x))
\end{align}
We calculate, using \eqref{GL-eq}
\begin{align}
  \label{eq:7}
 \int_{\widetilde \Omega} w_m(x) &\Big( |(\nabla - i \kappa H \Ab)
 f \psi |^2 - | \nabla f |^2 |\psi|^2
 \Big) \,dx
 +a\kappa^2\int_{\widetilde\Omega\setminus\Omega}|f\psi|^2\,dx\nonumber \\
&=
\kappa^2\int_{\widetilde \Omega}  \left(| \psi|^2
- |\psi|^4 \right)f^2\,dx\nonumber \\
&\leq \kappa^2 \int_{\Omega} f^2 |\psi|^2 \,dx
\end{align}
Now, using Lemma~\ref{lem:EnergyBelow} below, we can estimate
\begin{align}
  \label{eq:9}
  \int_{\widetilde \Omega} w_m(x) & |(\nabla - i \kappa H \Ab)
 f \psi |^2\,dx \nonumber \\
&= \int_{\Omega}  |(\nabla - i \kappa H \Ab)
 f \psi |^2\,dx
+
m^{-1} \int_{\widetilde \Omega \setminus \Omega}  |(\nabla - i \kappa H \Ab)
 f \psi |^2\,dx \nonumber \\
&\geq
\kappa H \big(1 - C/\sqrt{\kappa H}\big) \int_{\Omega} |f \psi |^2
\,dx\,.
\end{align}
Combining \eqref{eq:7} and \eqref{eq:9} we find
\begin{multline}
  \label{eq:10}
  ( \kappa H \frac{ b}{1+b} - C \sqrt{\kappa H} ) 
\int_{\Omega} | f \psi |^2 \,dx
+a\kappa^2\int_{\widetilde\Omega}|f\psi|^2\,dx
\leq
\int_{\widetilde \Omega}  |\nabla f|^2 |\psi|^2\,dx.
\end{multline}
We estimate the last term
\begin{align}
  \label{eq:11}
  \int_{\widetilde \Omega}  |\nabla f|^2 |\psi|^2\,dx
\leq   2 \varepsilon^2 \kappa H \int_{\widetilde \Omega} |f \psi|^2 \,dx
+
C \kappa H \int_{ \{\sqrt{\kappa H} t(x)\leq 1\}}
|\psi(x) |^2 \,dx.
\end{align}
Therefore we get, choosing $\varepsilon$ sufficiently small and for
$\kappa H$ sufficiently large,
\begin{align}
  \label{eq:12}
  \int_{\widetilde \Omega} |f \psi|^2 \,dx
\leq C \int_{ \{\sqrt{\kappa H} t(x)\leq 1\}}
|\psi(x) |^2 \,dx.
\end{align}
This implies the weighted $L^2$-bound in \eqref{exponentialdecay},
\begin{align}
  \label{eq:14}
  \int_{\widetilde \Omega} e^{2 \varepsilon \sqrt{\kappa H} t}
  |\psi(x)|^2 \, dx \leq C \int_{ \{\sqrt{\kappa H} t(x)\leq 1\}}
|\psi(x) |^2 \,dx.
\end{align}
Inserting \eqref{eq:14} in \eqref{eq:7} (and using the same
considerations) yields the weighted bound on $(\nabla - i \kappa H
\Ab) \psi$.
\end{proof}

\begin{lemma}\label{lem:EnergyBelow}
There exist constants $C_0, C_1$ such that if $(\psi, \Ab)$ is a
solution of \eqref{GL-eq} with $\kappa(H-\kappa) \geq C_0$, then
\begin{align}
  \label{eq:1}
  \| (\nabla - i \kappa H \Ab) \phi \|_{L^2(\Omega)}^2 \geq \kappa H(1
  - \frac{C_1}{\sqrt{H(H-\kappa)}}) \| \phi \|_{L^2(\Omega)}^2,
\end{align}
for all $\phi \in C^{\infty}_0(\Omega)$.
Also,
\begin{align}
  \label{eq:13}
   \| (\nabla - i \kappa H \Ab) \phi \|_{L^2(\widetilde \Omega \setminus \Omega)}^2 \geq \kappa H(1
  - \frac{C_1}{\sqrt{H(H-\kappa)}}) \| \phi \|_{L^2(\widetilde \Omega \setminus \Omega)}^2,
\end{align}
for all $\phi \in C^{\infty}_0(\widetilde \Omega \setminus \Omega)$.
\end{lemma}

\begin{proof}
We only prove \eqref{eq:1} the proof of \eqref{eq:13} being identical.

We estimate, using the compact support of $\phi$ and the standard
magnetic estimate from \cite[Thm 2.9]{AHS} (or \cite[Lemma~2.4.1]{FoHe07})
\begin{align}
  \label{eq:2}
   \| (\nabla - i \kappa H \Ab) \phi \|_{L^2(\Omega)}^2
&\geq \kappa H \int_{\Omega} \curl \Ab |\phi|^2\,dx\nonumber \\
&\geq \kappa H \| \phi \|_2^2 -(\kappa H) \| \curl \Ab - 1 \|_2 \| \phi \|_4^2.
\end{align}
By Lemma~\ref{standard} and the weak decay estimate of
Lemma~\ref{lem-BoFo}, we have
\begin{align}
  \label{eq:3}
  \| \curl \Ab - 1 \|_2 \leq \frac{C}{H \sqrt{\kappa(H-\kappa)}}.
\end{align}
Furthermore, by the Sobolev inequality and scaling followed by the
diamagnetic inequality, we find
\begin{multline}
  \label{eq:4}
  \| \phi \|_4^2 \leq C_{\rm Sob} \Big( \eta \| \nabla |\phi| \|_2^2 +
  \eta^{-1} \| \phi \|_2^2 \Big) \\
\leq  C_{\rm Sob} \Big( \eta \| (\nabla - i \kappa H \Ab) \phi \|_2^2
+ \eta^{-1} \|\phi \|_2^2 \Big),
\end{multline}
where $C_{\rm Sob}$ is a universal constant and $\eta>0$ is a
parameter that we can choose freely. We make the choice $\eta =
1/\sqrt{\kappa H}$. Combining \eqref{eq:2}, \eqref {eq:3} and
\eqref{eq:4} yields
\begin{align}
  \label{eq:5}
  \big(1 + \frac{C}{\sqrt{H(H-\kappa)}}\big)\| (\nabla - i \kappa H \Ab) \phi
  \|_{L^2(\Omega)}^2
  \geq
\kappa H \big(1 - \frac{C}{\sqrt{H(H-\kappa)}}\big) \| \phi \|_{L^2(\Omega)}^2,
\end{align}
from which \eqref{eq:1} follows.
\end{proof}

\section*{Acknowledgements}
The authors were supported by the European Research Council under the European Community's
Seventh Framework Programme (FP7/2007-2013)/ERC grant agreement
n$^{\rm o}$ 202859. SF is also supported by the Danish
Research Council and the Lundbeck Foundation.

\appendix
\section{Improved eigenvalue estimate for the disc}

The aim of this appendix is to prove Theorem~\ref{fe-disc}.
By assumption,
$\Omega=D(0,1)$ and $D(0,1+r)\subset\widetilde\Omega$ for some $r>0$.

Let $D(t) = \{ x \in {\mathbb R}^2 \,: \, |x| \leq t\}$
be the disc with radius $t$.
Let $\widetilde{Q}_B$
be the quadratic form
$$
\widetilde{Q}_B[u] = \int_{D(1+r) \setminus D(\frac{1}{2})}
\left(w_m(x)\big|(\nabla -iB\mathbf{F})u\big|^2+\alpha V_a(x)|u|^2\right)\,dx\;,
$$
with domain
$\{u \in H^1(D(1+r) \setminus D(\frac{1}{2})) \,| \, u(x) = 0 \text{ on
} |x|=\frac{1}{2}~{\rm and}~|x|=1+r  \}$. Here $w_m$ and $V_a$ are as in (\ref{w-V}), and
we emphasize that, for the sake of simplicity,
we omit the dependence on $a$, $m$ and $\alpha$ from the notation.

Let $\widetilde\mu_1(B,\alpha)$  be the lowest eigenvalue
of the corresponding self-adjoint operator.
Using the variational principle and the decay of the
ground state (Lemma~\ref{kach-decay}),  we see that,
\begin{align}
\label{eq:trekant}
\mu_1(B,\alpha) = \tilde{\mu}_1(B,\alpha) + {\mathcal O}(B^{-\infty})\;.
\end{align}

So, it is sufficient to prove \eqref{eq:18} with $\mu_1(B,\alpha)$
replaced by $\tilde{\mu}_1(B,\alpha)$.

By changing to boundary coordinates $(s,t)$ (defined by \eqref{eq:19})
the quadratic form $\widetilde{Q}_B[u]$ becomes,
\begin{align}
\widetilde{Q}_B[u] &= \int_0^{2\pi}  \int_{-r}^{1/2}
\widetilde w_m(t)\left[(1-t)^{-2}| (D_s - B \tilde{A}_1)u|^2+
|D_t u|^2\right](1-t) d  t ds\\
&\hskip0.5cm+
\int_0^{2\pi}  \int_{-r}^{1/2} \alpha \widetilde V_a(t)(1-t)|u|^2\, d
t ds\;, \\
\| u \|_{L^2}^2 &=  \int_0^{2\pi} \int_{-r}^{1/2} \,(1-t) |u|^2\,dtds\;,
\quad \tilde{A}_1 = \tfrac{1}{2} - t + \tfrac{t^2}{2}\;.\nonumber
\end{align}
Here
\begin{equation}\label{w-V*}
\widetilde w_m(t)=\left\{\begin{array}{l}
1\,,\quad{\rm if~}t>0\\
\displaystyle\frac1m\quad{\rm if~}t<0
\end{array}\right.
\quad \widetilde V_a(t)=\left\{
\begin{array}{l}
-1\,,\quad{\rm if~}t>0\\
a\,,\quad{\rm if ~}t<0\,.
\end{array}\right.
\end{equation}
Performing the scaling $\tau = \sqrt{B} t$ and decomposing in Fourier modes,
we find
\begin{align}
\label{eq:stjerne}
\tilde\mu_1(B,\alpha) = B \inf_{n \in {\mathbb Z}} e_{\delta(n,B), B}\;.
\end{align}
Here the function $\delta(m, B)$ was defined in \eqref{delta(a,m)}
and $e_{\delta,B}$ is the lowest eigenvalue of the quadratic form
$q_{\delta,B}$ on $L^2((-\sqrt{B}r,
\sqrt{B}/2);(1-\sqrt{B}\tau)d\tau)$ (with Dirichlet condition, $u(\tau)=0$, at
$\tau=-\sqrt{B}\,r$ and $\tau=\sqrt{B}/2$),
\begin{multline}
q_{\delta,B}[\phi]= \\
\int_{-\sqrt{B}r}^{\sqrt{B}/2}
\widetilde w_m(\tau)\left[(1-\tfrac{\tau}{\sqrt{B}})^{-1}
\big( (\tau-\xi ) +B^{-\frac{1}{2}}
(\delta -\tfrac{\tau^2}{2}) \big)^2\right.
+\left.(1-\tfrac{\tau}{\sqrt{B}}) |\phi'(\tau)|^2\right]
\, d \tau\\+
\alpha\int_{-\sqrt{B}r}^{\sqrt{B}/2}(1-\tfrac{\tau}{\sqrt{B}})\widetilde
V_a(t)|\phi(\tau)|^2\,
 d \tau\;.
\end{multline}
Here, $\xi =\xi(a,m,\alpha)$ by convention (this makes sense
provided that $m>m_*$, and
$\alpha\in[\alpha_0(a,m)-\epsilon_*(m),\alpha_0(a,m)+\epsilon_*(m)]$).
We will only consider $\delta$ varying in a fixed bounded set.
This is justified since it follows from
\cite[Proposition~4.6]{Kach} that for all $C>0$ there exists $D>0$
such that if $|\delta| > D$ and $B>D$, then
$$
e_{\delta,B} \geq \beta(a,m,\alpha)+C_1(a,m,\alpha) B^{-\frac{1}{2}} + CB^{-1}\;.
$$
Furthermore,  for
$\delta$ varying in a fixed bounded set,
we know (from the analysis of the operator (\ref{II-H-}),
especially that the minimum of (\ref{IIpVP}) in $\xi$ is non-degenerate)
that there exists a $d>0$ such that if $B>d^{-1}$, then
the spectrum of $q_{\delta,B}$ contained in $]-\infty, \beta(a,m,\alpha)+d[$
consists of exactly one simple eigenvalue.

The self-adjoint  operator
${\mathfrak h}(\delta,B)$ associated to
$q_{\delta,B}$ (on the space \break$L^2((-\sqrt{B}r,
\sqrt{B}/2);(1-\sqrt{B}\tau)d\tau)$) is
\begin{multline}
{\mathfrak h}(\delta,B)=
-(1-\tfrac{\tau}{\sqrt{B}})^{-1} \frac{d}{d\tau} \widetilde w_m(\tau)
(1-\tfrac{\tau}{\sqrt{B}}) \frac{d}{d\tau}\\
+\widetilde w_m(\tau)
\left[(1-\tfrac{\tau}{\sqrt{B}})^{-2} \big( (\tau-\xi )
  +B^{-\frac{1}{2}}(\delta -\tfrac{\tau^2}{2}) \big)^2\right]+\alpha
\widetilde V_a(\tau)\;.
\end{multline}
We will write down an explicit test function for ${\mathfrak
  h}(\delta,B)$ in \eqref{eq:trial} below, giving $e_{\delta,B}$ up to
an error of order ${\mathcal O}(B^{-\frac{3}{2}})$ (locally uniformly
in $\delta$).

We can formally develop ${\mathfrak h}(\delta,B)$ as
$$
{\mathfrak h}(\delta,B) = {\mathfrak h}_0 + B^{-\frac{1}{2}} {\mathfrak h}_1 + B^{-1} {\mathfrak h}_2 + {\mathcal O}(B^{-\frac{3}{2}})\;.
$$
with
\begin{align}
\label{eq:hs}
{\mathfrak h}_0 &= -\frac{d}{d\tau}\widetilde w_m(\tau)\frac{d}{d\tau} +\widetilde w_m(\tau)
  (\tau-\xi )^2+\alpha \widetilde V_a(\tau)\quad(=H[a,m,\alpha;\xi ])\;, \nonumber\\
{\mathfrak h}_1 &=\widetilde w_m(\tau)
\left[\frac{d}{d\tau}
+ 2 (\tau-\xi ) (\delta -\tfrac{\tau^2}{2})
+ 2\tau (\tau-\xi )^2\right]\;,\nonumber\\
{\mathfrak h}_2 &= \widetilde w_m(\tau)\left[
\tau\frac{d}{d\tau} +  (\delta -\tfrac{\tau^2}{2})^2
+ 4 \tau (\tau-\xi ) (\delta -\tfrac{\tau^2}{2})
+ 3 \tau^2 (\tau-\xi )^2\right]\;.
\end{align}
Let $u_0$ be the known ground state eigenfunction of
$H[a,m,\alpha;\xi ]$ with eigenvalue $\beta(a,m,\alpha)$.
Here, by $H[a,m,\alpha;\xi ]$, we mean the operator (\ref{II-H-})
with $\xi=\xi(a,m,\alpha)$, considered as a self-adjoint operator on
$L^2({\mathbb R}; d\tau)$. For ease of notation we will write
${\mathfrak h}_0$ instead of $H[a,m,\alpha;\xi ]$, since they are the
same formal differential operators.
Let $R_0$ be the regularized resolvent of ${\mathfrak h}_0$, which is defined by
$$
R_0 \phi = \begin{cases}({\mathfrak h}_0 -\beta(a,m,\alpha))^{-1}
  \phi\;, & \int \phi(\tau) u_0(\tau)\,d\tau = 0\;, \\ \quad 0\;, &
\phi \parallel u_0\,.
\end{cases}
$$
Let $\lambda_1$ and $\lambda_2$ be given by
\begin{align}
\lambda_1 &:= \langle u_0 \,|\,{\mathfrak h}_1 u_0 \rangle_{L^2({\mathbb R}; d\tau)}\;, &\nonumber\\
\lambda_2 &:= \lambda_{2,1} + \lambda_{2,2}\;,&\nonumber\\
\lambda_{2,1}&:= \langle u_0 \,|\,{\mathfrak h}_2 u_0 \rangle_{L^2({\mathbb R}; d\tau)} \;,&
\lambda_{2,2}&:= \langle u_0 \,|\,({\mathfrak h}_1 - \lambda_1) u_1 \rangle_{L^2({\mathbb R}; d\tau)} \;,
\end{align}
%
The functions $u_1, u_2$ are given as
\begin{align}
u_1 &= - R_0 ({\mathfrak h}_1 - \lambda_1) u_0\;,
&
u_2 &=- R_0 \big\{ ({\mathfrak h}_1 - \lambda_1) u_1
+ ({\mathfrak h}_2 - \lambda_2)u_0 \big\}\;.
\end{align}
Using the same type of argument in
\cite[Proposition~II.10]{Kach1} or \cite[Lemma~A.5]{FoHe04}, we can
prove that $R_0$ preserves exponential decay at infinity, i.e.
$u_0(t)$, $u_1(t)$, $u_2(t)$ and their derivatives decay
exponentially fast as $|t|\to\infty$.

Let $\chi \in C_0^{\infty}({\mathbb R})$ be a usual cut-off function, such that
\begin{align}
\chi(t) &= 1 \quad \text{ for } |t|\leq \tfrac{1}{8} \;, &
{\rm supp} \chi &\subset [-\tfrac{1}{4}, \tfrac{1}{4}]\;,
\end{align}
and let $\chi_B(\tau) = \chi(\tau B^{-\frac{1}{4}})$\,.\\
We define the following trial state,
\begin{align}
\label{eq:trial}
\psi := \chi_B \big\{ u_0 + B^{-\frac{1}{2}} u_1 + B^{-1} u_2 \big\}\;.
\end{align}
Using  the
exponential decay of the involved functions, we get after a
calculation,
\begin{align}
\big\| \big\{{\mathfrak h}(\delta,B) - \big(\beta(a,m,\alpha)
+ \lambda_1 B^{-\frac{1}{2}} + &\lambda_2 B^{-1}\big)\big\}
\psi \big\|_{L^2(]-\sqrt{B}/2, \sqrt{B}/2[;(1-\sqrt{B}\tau)d\tau)}\\
& = {\mathcal O}(B^{-\frac{3}{2}})\;,\nonumber
\end{align}
\begin{equation}
\| \psi \|_{L^2([-\sqrt{B}/2, \sqrt{B}/2[;(1-\sqrt{B}\tau)d\tau)} = 1 +{\mathcal O}(B^{-\frac{1}{2}})\;,
\end{equation}
where the constant in ${\mathcal O}$ is uniform for $\delta$ in bounded sets.
Applying the spectral theorem, and noticing that $\beta(a,m,\alpha)$
is an isolated eigenvalue for the operator $\mathfrak h_0$, we deduce
that (uniformly for $\delta$ varying in bounded sets),
\begin{align}
e_{\delta, B} = \beta(a,m,\alpha) + \lambda_1 B^{-\frac{1}{2}} + \lambda_2 B^{-1} +{\mathcal O}(B^{-\frac{3}{2}})\;.
\end{align}
It remains to calculate $\lambda_1, \lambda_2$ and, in particular,
deduce their dependence on $\delta$.

Writing,
$$2(\tau-\xi)\left(\delta-\frac{\tau^2}2\right)
+2\tau(\tau-\xi)^2=(\tau-\xi)^3-(\xi^2+2\delta)(\tau-\xi)\,,$$
and using
$$\int_{-\infty}^{\infty}
\widetilde w_m(\tau)(\tau-\xi ) u_0^2 \,d\tau = 0\,,$$
we get,
$$\lambda_1=C_1(a,m,\alpha)\,,$$
where $C_1(a,m,\alpha)$ is introduced in (\ref{II-C1}). In particular,
$\lambda_1$ is independent of $\delta$.

We do not need to calculate $\lambda_2$ 
explicitly. Notice that $\lambda_2(\delta)$ is a quadratic 
polynomial as a function of $\delta$.
We find the coefficient to $\delta^2$ as equal to,
$$ \int_{\mathbb R}\widetilde w_m(\tau)|u_0(\tau)|^2\, d \tau - 4 I_2\,,$$ with
\begin{align}
\label{eq:I2}
I_2 := \langle u_0 \, , \, \widetilde w_m(\tau)
(\tau - \xi ) R_0 \widetilde w_m (\tau - \xi ) u_0 \rangle\;.
\end{align}
Therefore, there exist constants $\delta_0, C_0 \in {\mathbb R}$ such that
$$
\lambda_2 = \left(\int_{\mathbb R}\widetilde
w_m(\tau)|u_0(\tau)|^2\, d \tau-4I_2\right)\,
\big( (\delta - \delta_0)^2 + C_0\big)\;.
$$
Recalling the definition of the constant $C_2(a,m,\alpha)$ in
\eqref{C2}, the above formula becomes,
$$
\lambda_2 = C_2(a,m,\alpha)\,
\big( (\delta - \delta_0)^2 + C_0\big)\;.
$$
In light of
\eqref{eq:trekant} and
\eqref{eq:stjerne}, we need only to show that $C_2(a,m,\alpha)>0$ to
finish the  proof  Theorem~\ref{fe-disc}.

Notice that we work under the hypothesis $m\geq m_*$. This implies that
$\xi$ is the unique, non-degenerate minimum of the function (see \eqref{IIpVP}
and \eqref{m*})
$$z\mapsto\mu(z):=\mu_1(a,m,\alpha;z)\,.$$
In particular,
$$\mu''(\xi)>0\,.$$
Now, exactly as shown in \cite[Proposition~A.3]{FoHe06}, it holds that
\begin{equation}\label{eq-c2=mu''}
C_2(a,m,\alpha)=\frac12\mu''(\xi)\,,\end{equation}
yielding thus the desired property regarding the sign of
$C_2(a,m,\alpha)$. 
This finishes the proof of Theorem~\ref{fe-disc}.\hfill$\Box$\\

For the sake of the reader's convenience, we include 
some details concerning the derivation of \eqref{eq-c2=mu''}.

\begin{proof}[Sketch of the proof of \eqref{eq-c2=mu''}.]
Let us introduce,
$$E(z)=\mu(z+\xi)\,,\quad H(z)=-\frac{d}{d\tau}\widetilde
w_m\frac{d}{d\tau}+\widetilde w_m(t)(t-\xi-z)^2+\alpha V_a(t)\,,$$
together with an analytic family of eigenfunctions $z\mapsto\phi(z)\in L^2(\R)$ such that
$$\|\phi(z)\|^2_{L^2(\R)}=1\,,\quad H(z)\phi(z)=E(z)\phi(z)\,,\quad
\phi(0)=u_0\,.$$
By differentiating the identity $\|\phi(z)\|^2=1$ twice with respect
to $z$, we get
\begin{equation}\label{eq-app-|phi|'(0)}
2{\rm Re}\langle \phi'(0)\,,\,u_0\rangle=0\,,\quad
{\rm Re}\langle \phi''(0)\,,\,u_0\rangle=-\|\phi'(0)\|^2
\,.
\end{equation}
Differentiating the relation $H(z)\phi(z)=E(z)\phi(z)$ we get since
$E(z)$ is minimal for $z=0$,
$$\left(H(0)-E(0)\right)\phi'(0)=2\widetilde w_m(t)(t-\xi)u_0\,.$$
Since the functions $\widetilde w_m(t-\xi)u_0$ and $u_0$ are
orthogonal in $L^2(\R)$, we get
\begin{equation}\label{eq-app-phi'(0)}
\phi'(0)=2\left(H(0)-E(0)\right)^{-1}(\widetilde
w_m\,(t-\xi)u_0)+c\,u_0\,,
\end{equation}
for some constant $c\in i\R$.

Differentiating twice the relation 
$E(z)=\langle \phi(z)\,,\,H(z)\phi(z)\rangle$, we get,
\begin{multline}\label{eq-app-e''(0)}
E''(0)=2E(0){\rm Re}\langle \phi'(0)\,,\,u_0\rangle-8{\rm Re}
\langle \phi'(0)\,,\,\widetilde w_m\,(t-\xi)u_0\rangle
\\+2{\rm Re}\langle\phi'(0)\,,\,H(0)\phi'(0)\rangle+2\langle
u_0\,,\,\widetilde w_m\,u_0\rangle\,.\end{multline}
Substituting \eqref{eq-app-|phi|'(0)} and \eqref{eq-app-phi'(0)} in
\eqref{eq-app-e''(0)}, we get,
$$E''(0)=2\left(\int_R\widetilde w_m(\tau)
  |u_0(\tau)|^2\,d\tau-4I_2\right)
$$
with $I_2$ introduced in \eqref{eq:I2}. Recalling the definition of
the constant $C_2(a,m,\alpha)$ we get the desired relation in 
\eqref{eq-c2=mu''}.
\end{proof}


\begin{thebibliography}{100}
\bibitem{ADN} {\sc
S. Agmon, A. Douglis, L. Nirenberg.}
Estimates near the boundary for solutions of
elliptic partial differential equations satisfying general
boundary conditions. I. {\it Comm. Pure Appl. Math.} {\bf 12} 623-727
(1959).
\bibitem{AHS} {\sc J.~Avron, I.~Herbst, and  B.~Simon.}
Schr\"odinger operators with magnetic fields I.
  General Interactions.
{\it Duke Math. J.} {\bf 45} 847-883 (1978).
\bibitem{BoFo} {\sc V. Bonnaillie-No\"el, S. Fournais.} Superconductivity in domains with corners.
{\it Rev. Math. Phys.} {\bf 19} (6) 607-637 (2007).
\bibitem{Chetal} {\sc S.J.~Chapman, Q. Du, M.D. Gunzburger}, A
Ginzburg Landau type model of superconducting/normal junctions
including Josephson junctions, {\it European J.~Appl. Math.} {\bf 6}
(2) 97-114 (1996).
\bibitem{DGP} {\sc Q. Du, M. Gunzburger, J. Peterson.}
Analysis and approximation of the Ginzburg--Landau model of
superconductivity. {\it SIAM Rev.} {\bf 34} 45-81 (1992).
\bibitem{Ev} {\sc L.C. Evans.} {\it Partial Differential Equations.}
Graduate Studies in Mathematics, 19.
American Mathematical Society, Providence, RI, 1998.
\bibitem{FoHe07} {\sc S. Fournais, B. Helffer.}
{\it Spectral Methods in Surface Superconductivity.} Monograph submitted.
\bibitem{FoHe} {\sc S. Fournais, B. Helffer.} On the third critical
field in Ginzburg-Landau theory. {\it Comm. Math. Phys.} {\bf 226}
(1) 153-196 (2006).
\bibitem{FoHe06} {\sc S. Fournais, B. Helffer.} Strong diamagnetism in general domains and
applications. {\it Ann. Inst. Fourier.} {\bf 57} (7) 2389-2400
(2007).
\bibitem{FoHe04} {\sc S. Fournais, B. Helffer.} Accurate eigenvalue
  asymptotics for the magnetic Neumann Laplacian. {\it
  Ann. Inst. Fourier.} (2006).
\bibitem{FoHe08} {\sc S. Fournais, B. Helffer.}
On the Ginzburg-Landau critical field in three dimensions. {\it
  Commun. Pure Appl. Math.} {\bf 62} 215-241 (2009).
\bibitem{FoKa} {\sc S. Fournais, A. Kachmar.}
Nucleation of bulk superconductivity close to critical magnetic
field. In preparation.
\bibitem{deGe} {\sc P.G. de\,Gennes}, {\it Superconductivity of metals and alloys,} Benjamin (1966).
\bibitem{deGe1} {\sc P.G. de\,Gennes}, Boundary effects in
superconductors, {\it Rev. Mod. Phys.} January 1964.
\bibitem{deGeHu} {\sc P.G. de\,Gennes, J.P. Hurault,} Proximity effects under magnetic fields~II- Interpretation
of `breakdown', {\it Phys. Lett.} {\bf 17} (3) 181-182 (1965).
\bibitem{Gi} {\sc T. Giorgi},
Superconductors surrounded by normal materials, {\it Proc. Roy. Soc.
Edinburgh,} {\bf 135A}  331-356 (2005).
\bibitem{GiPh} T.~Giorgi and   D.~Phillips.
\newblock The breakdown of superconductivity due to strong
fields for the Ginzburg-Landau
 model.
\newblock SIAM J.~Math.~Anal.~30  (1999), p.~341-359.
\bibitem{HeMo1} {\sc B. Helffer, A. Morame.}
Magnetic bottles in connection with superconductivity.
{\it J. Funct. Anal.} (2) 604--680 (2001).
\bibitem{HeMo2}{\sc B. Helffer, A. Morame.}
Magnetic bottles for the Neumann problem: curvature effect in the case
of dimension $3$ (general case). {\it Ann. Ec. Norm. Sup.} 37 (2004) p.~105-170.
\bibitem{HePa} {\sc B. Helffer, X.B. Pan.} Upper critical field and location of surface nucleation of
superconductivity. {\it Ann. Inst. H. Poincar\'e, analyse
non-lin\'eaire} {\bf 20} (1) 145-181 (2003).
\bibitem{Kach1} {\sc A. Kachmar.}
On the ground state energy for a magnetic Schr\"odinger operator and
the effect of the DeGennes boundary condition. {\it J. Math. Phys.}
{\bf 47} 072106 (32 pp.) (2006).
\bibitem{Kach2} {\sc A. Kachmar.} On the perfect superconducting
solution for a generalized Ginzburg-Landau equation. {\it Asymptot.
Anal.} {\bf 54} (3-4) (2007).
\bibitem{Kach} {\sc A. Kachmar.} On the stability of normal states
for a generalized Ginzburg-Landau model. {\it Asymptot. Anal.} {\bf
55} (3-4) 145-201 (2007).
\bibitem{KachTh} {\sc A. Kachmar.} {\it Probl\`emes aux limites issus de
la supraconductivit\'e, estimations semi-classiques et comportement
asymptotique des solutions.} Ph.D. Thesis, Universit\'e Paris-Sud
(2007).
\bibitem{Kato} {\sc T. Kato.} {\it Perturbation Theory for Linear
Operators.} Berlin: Springer-Verlag, 1976.
\bibitem{LuPa1} {\sc K. Lu, X.B.~Pan.}
Surface nucleation of superconductivity in $3$-dimensions. {\it
J. Differential Equations} {\bf 168} no.~2 (2000), 386-452.
\bibitem{LuPa} {\sc K. Lu, X.B. Pan.} Estimates of the upper critical field of the Ginzburg-Landau
equations of superconductivity. {\it Physica D} {\bf 127} no.~1-2,
73-104 (1999).
\bibitem{Pa} {\sc X.B. Pan.}
Surface superconductivity in applied magnetic fields above $H_{c2}$.
{\it Comm. Math. Phys.} {\bf 228} (2) 327-370 (2002).
\bibitem{SaSe} {\sc E. Sandier, S. Serfaty.} {\it Vortices in the Magnetic
Ginzburg-Landau Model.} Progress in Nonlinear Differential Equations
and Applications, Vol. 70, Birkhauser, 2007.
\bibitem{Se} {\sc S. Serfaty.} Local minimizers for the
Ginzburg-Landau energy near critical magnetic field. Part~I. {\it
Commun. Contemp. Math.} {\bf 1} (2) 213-254 (1999).
\bibitem{Se99}{\sc S. Serfaty.} Stable Configurations in Superconductivity :
Uniqueness, Multiplicity and Vortex-Nucleation.  {\it Archive for
Rational Mechanics and Analysis,} {\bf 149} 329-365 (1999).
\bibitem{Te} {\sc R. Temam.} {\it Navier-Stokes Equation. Theory and
  Numerical Analysis.} (Revised version)
Elsevier Science Publishers B.V. Amsterdam, 1984.

\end{thebibliography}
\end{document}